\documentclass[a4paper, 10pt, twosides]{amsart}

\usepackage{lmodern}
\usepackage{amsmath,amsthm}
\usepackage{amssymb}

\usepackage{mathtools}
\usepackage{graphicx}
\usepackage{caption}
\usepackage{subcaption}
\usepackage{enumitem}
\setenumerate{nolistsep}
\setitemize{nolistsep}
\usepackage[norelsize, ruled, vlined]{algorithm2e}
\usepackage[usenames, dvipsnames]{color}

\usepackage{booktabs}
\usepackage{epstopdf}
\usepackage{tikz}
\usetikzlibrary{decorations.markings}
\usetikzlibrary{shapes.geometric,positioning}
\usepackage[utf8]{inputenc}
 \usetikzlibrary{arrows}
\usetikzlibrary{automata, positioning}
\usepackage{url}
\usepackage[colorinlistoftodos, color=blue!20]{todonotes}

\usepackage[colorlinks=true,
		linkcolor=MidnightBlue,
		citecolor=Mahogany,
		urlcolor=ForestGreen,
		pdfauthor={Priyanka Dey},
		pdftitle={},
		pdfkeywords={},
		pdfsubject={Technical Report},
		plainpages=false]{hyperref}

\allowdisplaybreaks
\newtheorem{theorem}{Theorem}[section]

\newtheorem{lemma}[theorem]{Lemma}

\theoremstyle{definition}
\newtheorem{definition}[theorem]{Definition}

\theoremstyle{remark}
\newtheorem{remark}[theorem]{Remark}

\theoremstyle{assumption}
\newtheorem{assumption}[theorem]{Assumption}

\theoremstyle{fact}

\numberwithin{equation}{section}
\DeclareMathOperator*{\maximize}{maximize}
\DeclareMathOperator*{\minimize}{minimize}

\DeclareMathOperator{\argmax}{arg\,max}
\DeclareMathOperator{\sbjto}{subject\;to}

\newcommand{\abs}[1]{\left\lvert{#1}\right\rvert} 
\newcommand{\pmat}[1]{\begin{pmatrix}#1\end{pmatrix}}

\DeclareMathOperator{\rank}{rank}
\newcommand{\Real}{\mathbb{R}} 
\newcommand{\Nz}{\mathbb{N}} 
\newcommand{\N}{\Nz^\star} 
\newcommand{\Let}{\coloneqq}

\newcommand{\st}{x}

\renewcommand{\ge}{\geqslant}

\newcommand{\subfun}{\Xi}

\newcommand{\sys}{A}

\title[Efficient constrained sensor placement for observability]{Efficient constrained sensor placement for observability of linear systems}

\author[P. Dey]{Priyanka Dey}
\address{Systems \& Control Engineering, Indian Institute of Technology Bombay, Powai, Mumbai~400076, India.}
\email{dey\_priyanka@sc.iitb.ac.in}

\author[N. Balachandran]{Niranjan Balachandran}
\address{Department of Mathematics, Indian Institute of Technology Bombay, Powai, Mumbai 400076, India.}
\email{niranj@math.iitb.ac.in}

\author[D. Chatterjee]{Debasish Chatterjee}
\address{Systems \& Control Engineering, Indian Institute of Technology Bombay, Powai, Mumbai 400076, India.}
\email{dchatter@iitb.ac.in}

\begin{document}

\begin{abstract}
This article studies two problems related to observability and efficient constrained sensor placement in linear time-invariant discrete-time systems with partial state observations:

	\noindent (i) We impose the condition that both the set of outputs and the state that each output can measure are pre-specified. We establish that for any fixed \(k > 2\), the problem of placing the minimum number of sensors/outputs required to ensure that the structural observability index is at most \(k\), is NP-complete. Conversely, we identify a subclass of systems whose structures are directed trees with self-loops at every state vertex, for which the problem can be solved in linear time.
 
\noindent (ii) Assuming that the set of states that each given output can measure is given, we prove that the problem of selecting a pre-assigned number of sensors in order to maximize the number of states of the system that are structurally observable is also NP-hard. As an application, we identify suitable conditions on the system structure under which there exists an efficient greedy strategy, which we provide, to obtain a \((1-\frac{1}{e})\)-approximate solution. An illustration of the techniques developed for this problem is given on the benchmark IEEE 118-bus power network containing roughly \(400\) states in its linearized model.
\end{abstract}

\maketitle
\section{Introduction}
\label{s:intro}

The ever-increasing demand for low-cost control and quick reconstruction of past states from the observations for large-scale systems has brought to the foreground the problem of identifying a subset of the states with fewest elements that are required to ``efficiently'' control or observe these systems assuming exact knowledge of the system parameters. This problem may look deceptively simple, but it is a computationally difficult one. Indeed, \cite{ref:AOls-14} proves that finding the smallest number of actuators (resp. sensors) to make a linear system controllable (resp. observable) is NP-hard; an analogous statement holds for observability. 

Due to the sheer size and the ubiquitous modeling uncertainties of these systems, it is difficult to accurately survey the system parameters that govern their dynamics. Moreover, the parameters are prone to drift over time due to ageing, structural alterations, etc. Therefore, it is important to control large scale systems with the knowledge of only the interconnections among the various states of the dynamical system. This is still possible by using tools from structural systems theory that relies on the zero-nonzero pattern of the system matrices of a linear system, providing a fundamental bedrock on which conventional control theory may be enabled. Such an approach often leads to robust design procedures, especially in system with uncertain or unknown parameters, as in the case of power or biological networks.  A survey on various optimization problems studied via structural system theory can be found in \cite{ref:LiuBarAlb-16}. In view of structural observabilty, a collection of interesting problems have been recently addressed in \cite{ref:ComDioTri-08}, \cite{ref:DooKha-13}, \cite{ref:DooRabZarKha-17} and the references therein. These works mainly focus  on optimal placement of sensors to ensure structural observability via efficient polynomial time algorithms and classification of sensors according to the influence of their failure on observability. 

Since a significant fraction of real world systems admit only partial state observations, one of the central problems in systems theory is the efficient recovery of the actual system states from the observations. Moreover, for many networks it is of practical importance to recover the states quickly and efficiently within a small time window. Thus, it is crucial to understand \textit{how quickly} the states can be recovered from the observations of a discrete-time linear system. The \emph{observability index} characterises this speed of recovery by determining the minimum number of iterations required to fully reconstruct the states of a discrete-time linear system. Due to inevitable system uncertainties, we focus on the structural counterpart of the observability index, namely, the structural observability index. We address the following problem:
\vspace*{-0.15cm}
 \begin{quote} \textit{Minimal Sensor Placement Problem:} Determining the minimal number and placement of sensors/outputs  required to guarantee a desired bound on the structural observability index when the given output matrix has a rectangular diagonal structure, i.e., the state to which each output is directly connected is
pre-specified.\footnote{Note that the given output matrix can also be a diagonal structured matrix.}\end{quote} 
Let \(d\) be the number of states in the linear system. 
Corresponding to this problem, we have the following results:
\vspace*{-0.2cm}
\begin{itemize}[label=\(\circ\), leftmargin=*]
\item We provide an algorithm with run time complexity \(O(d^3 \log d)\) to determine a solution of the minimal sensor placement problem when the desired bound on the structural observability index is equal to \(2\). 
\item We prove that, in the general set up, the minimal sensor placement problem is NP-hard whenever the desired bound on the structural observability index is at least \(3\), thereby illustrating a sharp transition in the hardness of the problem as the bound changes from 2 to 3. 
\end{itemize}
Since the general problem turn out to be difficult, we identify conditions under which the minimal sensor placement problem is polynomially solvable. We consider a practically relevant special class of systems whose structure is a directed tree with a self-loop at every state.\footnote{See \S\ref{s:Preliminaries} for a formal definition of the directed tree. This particular structure plays an important role for a large class of systems including leader-follower networks \cite{ref:PeqKruKarMouPed-13}, \cite{ref:KruPeqKarMouAgu-18}, biological networks \cite{ref:AkuHayChiNg-07}, transportation systems \cite{ref:ZheLiBorHed-17}, \cite{ref:WanSyeYinPanZha-14}, etc. Furthermore, in the multi-agent community, several well-known strong results \cite{ref:RenBea-05}, \cite{ref:RenBeaAtk-07} rely on the existence of spanning trees in the network, thereby conforming to this category of structural hypotheses.} We establish the following result for this subclass:
\vspace*{-0.15cm}
\begin{itemize}[label=\(\circ\), leftmargin=*]
\item The minimal sensor placement problem is solvable in polynomial time when the underlying graph structure of the system is a directed tree with a self-loop at every state and we give an \(O(d)\) algorithm to solve it.
\end{itemize}
\vspace*{-0.1cm}

We move to the second problem addressed in this article:\vspace*{-0.15cm} \begin{quote} \textit{Cardinality Constrained Sensor Placement Problem:} Identifying a pre-assigned number of outputs from the given set of outputs so that maximum number of states are \emph{structurally observable by them}\footnote{See \S\ref{s:cardinality constrained sensor placement} for a formal definition.}  in the system when the set of states that  each output can measure is pre-specified. \end{quote} It becomes extremely relevant when the permissible number of sensors/outputs may not be adequate to observe the entire system. Therefore, a design strategy to select the sensors/outputs in such a way that a maximal number of states are observable in the system is needed. Corresponding to this problem, we have the following result:
\vspace*{-0.15cm}
\begin{itemize}[label=\(\circ\), leftmargin=*]
 \item We establish that the cardinality constrained sensor placement problem is NP-hard. We observe that the problem remains NP-hard even if we impose a mild condition on the system where the digraph associated with the state matrix \(A\) is such that each state vertex has a self-loop.\footnote{The systems satisfying this mild condition that every state vertex has a self-loop are often referred as \emph{self-damped systems} in the literature.}
\end{itemize} 
\vspace*{-0.1cm}
We confine ourselves to this special class where each state vertex has a self-loop since a wide class of systems exhibit this self-damped dynamics \cite{ref:ChaMes-13}, for example, epidemic spread in networks \cite{ref:NowPrePap-16}, ecosystems \cite{ref:May-01}, power grids \cite{ref:IliXieKhaMou-10}, and even social networks \cite{ref:PeqKarAgu-14}. For this class of systems, we give the following result:
\vspace*{-0.15cm}
\begin{itemize}[label=\(\circ\), leftmargin=*]
\item We provide a greedy algorithm to obtain a \((1-1/e)\)-approximate solution for the cardinality constrained sensor placement problem. This is the \emph{best possible result} that can be obtained via greedy algorithms and at the level of generality considered here.
\end{itemize} 

The rest of this article unfolds as follows. \S\ref{s:related work} discusses a few existing results and related work in this area. \S\ref{s:Preliminaries} reviews certain useful concepts and results. The precise problem statements of the minimal sensor placement problem and the cardinality constrained sensor placement problem, and our corresponding main results are given in \S\ref{s:Minimal sensor placement problem} and \S\ref{s:cardinality constrained sensor placement} respectively. We conclude a summary of our results  along with possible future directions in \S\ref{s:conclusion}.
\section{Related work} 
\label{s:related work}
The definition of structural observability index was introduced in \cite{ref:HMor-82}, and a few methods required for its computation were proposed in \cite{ref:SueDau-97}. By employing graph-theoretic techniques bounds on the (controllability and) observability index for structured linear systems were provided in \cite{ref:SunHad-13}. \cite{ref:SerAlbGeo-17} considered the problem of identifying the minimum number of states to be connected to distinct inputs (resp. outputs) to ensure a given bound on the structural controllability (resp.\ observability) index, and established that the problem is NP-hard. In addition, the trade-off between the structural controllability index and the minimum number of states that need to be actuated was explored on a variety of artificial and synthetic networks by using a heuristic algorithm  and it was observed that the number of actuated states obtained is close to optimal. The problem we address is a generalization of the problem considered in \cite{ref:SerAlbGeo-17}: the selection of states to be measured by distinct outputs 
is \emph{constrained} to a specific preassigned family of states. This restriction makes the problem more realistic and increases the level of its difficulty; 
see Remark \ref{r:NP-hardp1} for further technical details. In addition, we identify a practically relevant subclass of systems for which the minimal sensor placement problem is \emph{optimally} solvable, viz., systems whose structures are directed trees with a self-loop at each vertex.

In control theory, several problems involving selection of a pre-specified/minimum number of states (or inputs/outputs) to optimize a certain objective function have been studied via submodularity tools and can be found in \cite{ref:ClaAloBusPoo-16}. In particular, the problem of identifying a pre-assigned number of states to be actuated in order to optimize some of the energy metrics, such as the trace of the controllability Grammian, the log of determinant of the controllability Grammian, etc., was investigated in  \cite{ref:SumCorLyg-16} by using the notion of submodularity. Another work where the problem of selecting the minimum number of sensors so as to optimize the Kalman filter with respect to the estimation error in a linear time-invariant system corrupt with process and measurement noise was studied employing submodularity in \cite{ref:TzoJadPapp-16}. In the context of our problem, a relaxed version of the cardinality constrained placement problem (in the controllability framework) where each input (resp. output) is directly connected to only one state was dealt in \cite{ref:ClaBusPoo-12} via submodularity when the digraph associated with the state matrix \(A\) is strongly connected.\footnote{A digraph \(G=(V, E)\) is \emph{strongly connected} if for each ordered pair of vertices \((x_i, x_j)\), \(G\) has a directed path from \(x_i\) to \(x_j\).} However, they do not comment on the complexity of this problem. In contrast to \cite{ref:ClaBusPoo-12}, we address this problem in a different subclass where it is difficult to solve and there is no restriction on the given outputs to measure only one state.
\section{Preliminaries}
\label{s:Preliminaries}
The notations used in this article are standard: The set of real numbers, non-negative integers, and positive integers are denoted by \(\Real\), \(\mathbb{N}\), and \(\N\) respectively. Let \([r]\Let\lbrace 1,2,\ldots, r\rbrace\) for  each \(r\in \N\). The cardinality of set \(X\) is denoted by \(\abs{X}\). For sets \(X\) and \(Y\), \(X \setminus Y\) represents the elements belonging to \(X\) and not to \(Y\). For a matrix \(A\) of appropriate dimension, \(A_{ij}\) or \([A]_{ij}\) represents the \((i,j)\)-entry of this matrix.

Consider a linear time-invariant system
\begin{equation}
\begin{aligned}
	\label{e:linsys}
	\st(t+1)&= \bar{\sys} \st(t),\,\, y(t)= \bar{C} \st(t), \quad t\in \mathbb{N},  
\end{aligned}
\end{equation}
where \(\st(t) \in \Real^{d}\) and \(y(t)\in \Real^{p}\) are the state and output vector at time \(t\). The state and output matrices are given by \(\bar{\sys} \in \Real^{d \times d}\) and \(\bar{C}\in \Real^{p\times d}\) respectively. In this article, the system \eqref{e:linsys} is sometimes described by the pair \((\bar{A},\bar{C})\). In our analysis only the information about the locations of the fixed zeros in \(\bar{A}\) and \(\bar{C}\) is crucial and the precise numerical values of the non-entries of \(\bar{A}\) and \(\bar{C}\) are not relevant. For any matrix \(H\), its \emph{sparsity matrix} is defined as a matrix of the same dimension as \(H\) with either a zero or an independent free parameter (denoted by \(\ast\)) at each entry depending on whether the corresponding entry in \(H\) is zero or not. A \emph{numerical realisation} of a sparsity matrix is obtained by giving numerical values to its star entries. Let the sparsity matrices of the state and the output matrices in \eqref{e:linsys} are represented by \(A\in \lbrace 0,\ast \rbrace^{d \times d}\) and \(C\in \lbrace 0,\ast \rbrace^{p \times d}\) , and let \([A]\) and \([C]\) be the collection of all numerical matrices of the same dimension and structure/sparsity as \(A\in \lbrace 0,\ast \rbrace^{d \times d}\) and \(C\in \lbrace 0,\ast \rbrace^{p \times d}\) respectively. We say that a pair \((A, C)\) is \emph{structurally observable} if there exists at least one observable numerical realization of \((A, C)\).\footnote{It is known that if one realization of \((A, C)\) is observable, then \emph{almost all} numerical realizations of \((A, C)\) are observable; see \cite{ref:Lin-74}.}

A linear time-invariant system \eqref{e:linsys} is associated with a digraph \(G(A,C)\) by using the following natural way: Let \(\mathcal{A}=\lbrace \st_1, \st_2, \ldots,  \st_d\rbrace\) and \(\mathcal{C}=\lbrace 1, 2, \ldots, p\rbrace\) be the state vertices and the output vertices corresponding to the states \(\st(t)\in \Real^d\) and the outputs \(y(t) \in \Real^p\) of the system \eqref{e:linsys}. Let \( E_{A}=\lbrace (\st_j, \st_i)\,|\, {A}_{ij}\neq 0\rbrace\) and \(E_{C}=\lbrace (\st_j, i)\,|\, C_{ij}\neq 0\rbrace\). The digraph \(G(A,C)=(\mathcal{A}\sqcup \mathcal{C}, E_{A} \sqcup E_{C})\), and \(\sqcup\) denotes the disjoint union. The sets \(E_{A}\) and \(E_{C}\) denote the edges between the state vertices, and the edges from the state vertices to the output vertices in the digraph \(G(A,C)\) respectively. Similarly, we can define digraph \(G(A)= (\mathcal{A}, E_{A})\) with vertex set \(\mathcal{A}\) and edge set \(E_{A}\). In \(G(A,C)\), the \emph{induced subgraph} by \(U\subset\mathcal{A} \sqcup \mathcal{C}\) is a digraph consisting of vertex set \(U\) and all those edges of the digraph \(G(A, C)\) with both end points in \(U\). In particular, \(G(A)\) is the induced subgraph of \(G(A,C)\) by \(\mathcal{A}\). 

A sequence of edges \(\lbrace(\st_1, \st_2), (\st_2, \st_3), \ldots, (\st_{k-1}, \st_k)\rbrace\), where each \(\st_i \in \mathcal{A}\) is distinct and \((\st_i, \st_j)\in E_{A}\), is called a \emph{directed path} from \(\st_1\) to \(\st_k\) in \(G(A)\). A \emph{cycle} is a directed path where the initial vertex \(\st_1\) coincides with the end vertex \(\st_k\). A digraph is acyclic if it contains no cycles. A digraph is a \emph{directed tree towards \(\st\)} if it is an acyclic graph where every vertex has a directed path towards \(\st\) and every vertex except \(\st\) has out-degree exactly equal to \(1\). Sometimes we refer to this digraph as just directed tree. The vertices with no incoming edges are termed as the \emph{leaves} of the tree. The digraph \(G(A, C)\) is said to have a \emph{spanning forest topped} at output vertices \(\mathcal{C}\) if it has a disjoint union of set of directed trees,  where each tree is directed towards a vertex in \(\mathcal{C}\) and the union contains all the state vertices.
\vspace*{-0.1cm}
\begin{definition} 
For the digraph \(G(A,C)\) associated with system \eqref{e:linsys} and
a subset \(S \subset \mathcal{A}\), the \emph{out-neighbourhood} of \(S\) is the set \(N^{+}(S)=\big\lbrace v\, \big|\, (\st_i, v) \in E_{A}\sqcup E_{C}, \st_i \in S, v\in \mathcal{A} \sqcup \mathcal{C}\big \rbrace\). The digraph \(G(A,C)\) is said to have a \emph{contraction} if  there exists a set \(S \subset \mathcal{A}\) with \(\abs{N^{+}(S)} < \abs{S}\).
\end{definition} 

The following subgraphs associated with digraph \(G(A)\) and \(G(A, C)\) are defined in \cite{ref:PeqKarAgui-16}.

\begin{itemize}[label=\(\circ\), leftmargin=*]
\item State stem is a directed path, consists of only state vertices. An isolated state vertex is also considered as state stem. \footnote{The tip of a state stem is a state vertex that does not have any outgoing  edges from it to any other state vertex in the stem.}
\item Output Stem is a directed path obtained by connecting a directed edge from the tip of a state stem to an output vertex. 
\item An Output Cactus, defined recursively as follows: An output stem with at least one state vertex is an output cactus.  An output cactus connected by a directed edge from a state vertex of a disjoint cycle (comprising of only state vertices) to either any state vertex or the output vertex of the cactus is also an output cactus.
\end{itemize}
The relation between certain properties of \(G(A, C)\) and the structural observability of the pair \((A,C)\) is given by:

\begin{theorem}\cite{ref:Lin-74}\cite{ref:DooKha-13}
\label{t:observability}
The following are equivalent:
\begin{enumerate}[label={\rm (\alph*)}, widest=b, leftmargin=*, align=left]
\item The pair \((A,C)\) is structurally observable.
\item In the digraph \(G(A,C)\) derived from \eqref{e:linsys}, every state vertex \(\st_i\) has a directed path from it to at least one of the output vertices, and \(G(A,C)\) is free of contractions.
\item The digraph \(G(A, C)\) is spanned by a disjoint union of output cacti.
\end{enumerate}
\end{theorem}
\vspace*{-0.3 cm}
 \subsection{Structural observability index} 

 We catalogue some notions specific to structural observability index. The \emph{observability index} \(\mu(\bar{A},\bar{C})\) of \eqref{e:linsys} is \footnote{The convention that \(\inf\emptyset=+\infty\) is assumed to be in place.}
\begin{equation*}
\label{e:obsind}
\mu (\bar{A},\bar{C})\Let\inf\Big\lbrace k \in [d]\, \Big| \, \rank\pmat{\bar{C}^{\top} & (\bar{C}\bar{A})^{\top} & \cdots & (\bar{C}\bar{A}^{k-1})^{\top} }^{\top}=d\Big\rbrace.
\end{equation*}
In other words, \(\mu(\bar{A},\bar{C})\) is the minimum number of iterations required to recover/determine uniquely the initial state \(x_0\) from \(y_0\), \(y_1,\ldots\). In other words, \(x_0\) may be obtained by left-inversion in the linear equation
\[\begin{pmatrix}
\bar{C}\\
\bar{C}\bar{A}\\
\vdots\\
\bar{C}\bar{A}^{\mu(\bar{A},\bar{C})-1}
\end{pmatrix}x_0=\begin{pmatrix}
y_0\\
y_1\\
\vdots\\
y_{\mu(\bar{A},\bar{C})-1}
\end{pmatrix}.
\]
The \(k\)-step observability matrix associated with the pair \((\bar{A},\bar{C})\) is given by \(O_k(\bar{A},\bar{C})\Let\begin{pmatrix}
\bar{C}^{\top} & (\bar{C}\bar{A})^{\top} & \ldots & (\bar{C}\bar{A}^{k-1})^{\top}
\end{pmatrix}^{\top}\)
The structural counterpart of the observability index, namely, \emph{structural observability index}, is defined as

\begin{equation}
\mu(A,C)\Let\inf\Biggl\lbrace k\in [d] \,\Bigg| \,\sup_{\substack{A_1 \in [A]\\C_1 \in [C]}}
\rank\pmat{O_k(A_1, C_1)}=d\Biggr\rbrace.
\end{equation}
The value of the infimum is \(+\infty\) when none of the pairs in \(([A],[C])\) is observable, i.e., the pair \((A, C)\) is structurally unobservable. If the pair \((A, C)\) has structural observability index \(\mu(A, C)=\ell\), where \(\ell\leq d\) is some positive integer, then almost all numerical realisations of pair \((A,C)\) with the same structure as \((A, C)\) have observability index \(\ell\) off a manifold with zero Lebesgue measure \cite[p.~44]{ref:Rei-88}. The following result provides a graph theoretic interpretation of the structural observability index of a pair \((A, C)\).
\vspace*{-0.1 cm}
\begin{theorem}\cite[Theorem 2]{ref:SerAlbGeo-17} 
\label{t:observabilityindex}
A pair \((A, C)\) is structurally observable with index \(\ell\) if and only if  the digraph \(G(A, C)\) is spanned by a disjoint union of output cacti, where every output cactus contains at most \(\ell\) state vertices.  
\end{theorem}
\vspace*{-0.1cm}
The graph-theoretic notion of structural observability index is demonstrated via a digraph \(G(A, C)\) shown in Fig.~\ref{f:figstructuralindex}. Given \(A \in \lbrace 0, \ast \rbrace^{d\times d}\) and \(C\in \lbrace 0, \ast\rbrace^{p\times d}\), there exists several possible disjoint unions of output cacti spanning the digraph \(G(A, C)\). 
 
\tikzset{middlearrow/.style={
        decoration={markings,
            mark= at position 0.5 with {\arrow{#1}} ,
        },
        postaction={decorate}
    }
}
\begin{figure}[th]
\begin{subfigure}{.3\textwidth}
  \includegraphics[width=\textwidth]{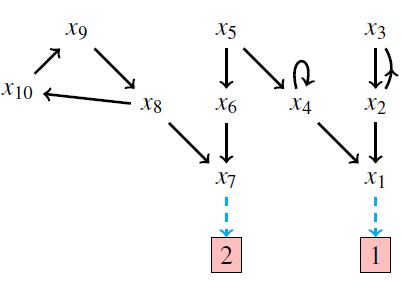}
 \caption*{(a)}
\end{subfigure}
\quad
\begin{subfigure}{.3\textwidth}
  \includegraphics[width=\textwidth]{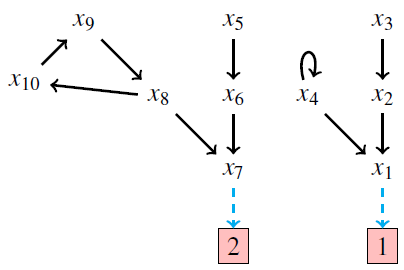}
\caption*{(b)}
\end{subfigure}
\quad
\begin{center}
\begin{subfigure}{.3\textwidth}
 \centering
 \includegraphics[width=\textwidth]{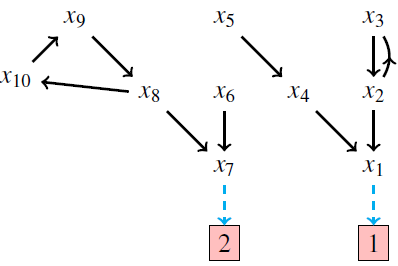}
\caption*{(c)}
\end{subfigure}
\end{center}
\caption{Illustration of a digraph \(G(A, C)\) in (a) and  two possible output cacti shown in (b) and (c). Observe that (b) has four state vertices in one output cactus and six in the other, whereas (c) has one output cactus with five state vertices and the other with five. Moreover, it is not hard to show that these  two are the only spanning output cacti of \(G(A, C)\); Therefore, the structural observability index is five.}
\label{f:figstructuralindex}
\end{figure}

\begin{lemma}\cite[Corollary 1]{ref:SerAlbGeo-17} 
\label{l:observabilityindex}
Let \(A\in \lbrace 0, \ast \rbrace^{d \times d}\) and \(C \in \lbrace 0, \ast \rbrace^{p \times d}\), with \(A_{ii}=\ast\) for all \(i=1,2, \ldots, d\). Let \(\lbrace \mathcal{H}_i\rbrace_{i \in \mathcal{K}}\) be the collection of all spanning forests of \(G(A, C)\) containing only directed trees towards the output vertices, where \(\mathcal{K}\) contains the indices of such spanning forests. Let  \(z_i \in \N\)  be the number of directed trees in \(\mathcal{H}_i\), i.e., \(\mathcal{H}_i= \lbrace \mathcal{T}_j^i\rbrace_{j=1}^{z_i}\). Then, the structural observability index of the pair \((A, C)\) is
\[
\mu(A, C)\Let\min_{i\in \mathcal{K}}\max_{\mathcal{T}\in \mathcal{H}_i}\abs{\mathcal{T}}_s,
\]
where \(\abs{\mathcal{T}}_s\) is the number of state vertices in a tree \(\mathcal{T}\) towards an output vertex.
\end{lemma}
\subsection{Submodular function}
We need a short digression into the properties of submodular functions:
Let \(V\) be a non-empty finite set, and let \(2^V\) denote the collection of all subsets of \(V\).
\begin{definition}
\label{d:submodular}
 A function \(f:2^V\to \Real\) is \emph{submodular} if for all \(S\subset T\subset V\)  and \(v\in V\setminus T\) we have \(f(S\cup \lbrace v\rbrace)-f(S)\geq f(T\cup \lbrace v \rbrace)-f(T).\)
\end{definition} 
A function \(f:2^V\to \Real\) is \emph{supermodular} if \((-f)\) is submodular.
\begin{definition}
A function \(f:2^V\to \Real\) is \emph{monotone non-decreasing} if for every \(S\subset T\subset V\) we have \(f(S)\leq f(T)\).
\end{definition}

In words, the property of submodularity implies that adjoining an element \(v\) to the set \(S\) results in a larger increment in \(f\) than adding \(v\) to a superset \(T\) of \(S\)  \cite{ref:Lov-83}. When a function \(f\) is submodular, one can use greedy algorithms that yield, in reasonable time \cite{ref:NemWolFis-78}, approximate solutions that are often very close to an optimal solution.
\section{Minimal sensor placement problem}
\label{s:Minimal sensor placement problem}
Before stating the precise problem statement, we define the identity structured output matrix \(C=I_{d}\in \lbrace 0, \ast \rbrace^{d\times d}\) as follows: \([I_d]_{ij}=\ast\) if \(i=j\) and \(0\) otherwise, where \(i,j\in [d]\). Recall \(\mathcal{A}=\lbrace \st_1, \st_2, \ldots, \st_d\rbrace\). An output matrix \(C=I_F \in \lbrace 0, \ast \rbrace^{\abs{F}\times d}\) is obtained from \(I_{d}\), as defined above, by retaining the rows corresponding to the state vertices in \(F\subset \mathcal{A}\).

Given \(A\in \lbrace 0, \ast \rbrace^{d \times d}\), \(C=I_F\in \lbrace 0, \ast \rbrace^{\abs{F}\times d}\), and \(\ell \in [d]\) such that \(\mu(A,I_F)\leq \ell\),
\begin{equation}
\label{keyproblem1}
\left\{
\begin{aligned}
   &\minimize_{J\subset F} && \abs{J}\\
   &\sbjto && \mu(A, I_J)\leq \ell,
\end{aligned}
\right.
\tag{\(\mathcal P_1\)}
\end{equation}
where \(F \subset \mathcal{A}\), \(I_J\) is the submatrix of the output matrix \(I_F\) obtained by retaining the rows corresponding to state vertices in \(J\). \eqref{keyproblem1} determines the minimal subset of states required to admit outputs from the set of available states \(F\) to ensure a desired bound on structural observability index.

A special case of \eqref{keyproblem1} is when the output matrix \(C=I_{d}\in \lbrace 0, \ast \rbrace^{d \times d}\), i.e., \(F=\mathcal{A}\) and we refer to this problem as \((\mathcal P_1^\prime)\) in this sequel.

Next we show that \eqref{keyproblem1} (and hence \((\mathcal P_1^\prime)\)) can be solved optimally when the bound \(\ell\in \lbrace 1, 2\rbrace\). Clearly, solving \eqref{keyproblem1} for \(\ell=1\) is trivial since it requires each state vertex to be directly connected to a distinct output, i.e., the solution is the output matrix \(C\) of identity structure \(I_{d}\) (as defined above).

The concept of a matching is needed to solve \eqref{keyproblem1} for \(\ell=2\). Recall that for an undirected graph \(G=(V, E)\), a \emph{matching} is a set of edges with no shared endpoints. The vertices incident to the edges of a matching \(M\) are said to be \emph{saturated} by \(M\); the other are \emph{unsaturated}. A \emph{maximum matching} is  one that has the largest cardinality among all possible matchings in a graph. We define a weight function \(w: E \to \Real\) that assigns weights to the edges of the graph \(G\). Subsequently, we introduce the \emph{maximum weight matching} problem, concerned with finding a matching of \(G\) that has the maximum weight-sum of its edges; in other words, determining a matching \(M^{*}\) such that \(\sum_{e \in M^{*}} w(e) \geq \sum_{e \in \bar{M}} w(e)\) for any matching \(\bar{M}\) in \(G\). We solve \eqref{keyproblem1} for \(\ell=2\) by employing Algorithm \ref{AlgoProblem1}: 

\begin{algorithm}[h]
	\KwIn{\(G(A)=(\mathcal{A}, E_A)\), \(F \subset\mathcal{A}\) such that \(\mu(A, I_F)\leq 2\).}
\KwOut{A solution of \eqref{keyproblem1},  \(J^*\subset F\).}
\nl Construct an undirected graph \(G_u=(V, E_u)\) from \(G(A)\) with \(V=\mathcal{A}\) and \(E_u= E_F \sqcup E_{\mathcal{A}\setminus F}\). \((\st_i, \st_j) \in E_F\), if \(\st_i, \st_j \in F\) and either \((\st_i,\st_j)\in E_A\) or \((\st_j, \st_i)\in E_A\). \((\st_k,\st_m)\in E_{\mathcal{A}\setminus F}\), if \(\st_k \in \mathcal{A}\setminus F , \st_m\in F\) and \((\st_k,\st_m)\in E_A\). Neglect self-loops, if present.

\nl Define the weight \(w\) for the edges of \(G_u\):

   \begin{equation}
\label{e:weight}
w(e)\leftarrow\begin{cases} 1 &\mbox{for}\,\, e \in E_F,\\
\abs{E_u}+1 &\mbox{for}\,\, e\in E_{\mathcal{A}\setminus F}.
\end{cases}   
\end{equation}

\nl  Determine a maximum weight matching \(M\) in \(G_u\) under weight \(w\).

\nl For each edge \(e=(\st_i, \st_j)\in M\), the direction of the edge \(e\) is chosen in accordance to the direction of \(e\) in \(G(A)\).

\nl If \(e=(\st_i, \st_j)\in M \) with \((\st_i, \st_j)\in E_A\) then \(\st_j \in J^*\) and if \(\st_k\) is unsaturated by \(M\) then \(\st_k \in J^*\). 
 
\caption{Algorithm to solve \eqref{keyproblem1} for \(\ell = 2\).} 
\label{AlgoProblem1}
\end{algorithm}

In Step \(1\), we construct an undirected graph \(G_u\) from \(G(A)\) with vertex set \(\mathcal{A}\) and edge set \(E_u=E_F\sqcup E_{\mathcal{A}\setminus F}\). An edge exists between a pair of vertices \(\st_i\), \(\st_j\) in \(G_u\) if any one of the conditions is satisfied: (i) both the vertices lie in \(F\) and either \((\st_i, \st_j)\in E_A\) or \((\st_j, \st_i)\in E_A\); (ii) one vertex say \(\st_i\) lies in \(\mathcal{A}\setminus F\) and the other one \(\st_j\) is in \(F\) and a directed edge exists from \(\st_i\) to \(\st_j\) in \(G(A)\). The undirected edges satisfying condition (i) forms set \(E_F\) and the undirected edges satisfying condition (ii) forms set \(E_{\mathcal{A}\setminus F}\). We assign weights to the edges of \(E_u\) and compute a maximum weight matching \(M\) in \(G_u\) in Steps \(2\) and \(3\) respectively. In Step \(5\), we form a set \(J^{*}\) by collecting the vertices unsaturated by \(M\) and by selecting one vertex from each undirected edge in \(M\) in the following way: for \(e\)=\((\st_i, \st_j)\in M\), if \((\st_i, \st_j)\in E_A\) then \(\st_j\in J^{*}\); if \((\st_j, \st_i)\in E_A\) then \(\st_i\in J^{*}\); if both \((\st_i, \st_j)\) and \((\st_j, \st_i)\) are in \(E_A\) then any one among \(\st_i\) or \(\st_j\) is collected in \(J^{*}\). 

The construction of the undirected graph \(G_u\) from \(G(A)\) requires \(O(d^2)\) computations. A maximum weight matching is computed in \(O(d^3\log d) \) computations \cite[Chapter 11]{ref:BerVyg-06} and the rest of the steps have linear complexity. Therefore, the overall complexity of  Algorithm \ref{AlgoProblem1} is \(O(d^3\log d)\).

\begin{lemma}
\label{l:lemmasoi2}
Let \(A \in \lbrace 0, \ast\rbrace^{d\times d}\), \(C\)=\(I_F\in \lbrace 0, \ast \rbrace^{\abs{F}\times d}\), and \(G(A)=(\mathcal{A}, E_A)\) be the state matrix, the output matrix, and the digraph associated with \(A\) and \(d\geq 2\). Suppose \(\mathcal{A}\setminus F\neq \emptyset\) and \(G_u\) be the undirected graph associated with \(G(A)\) obtained via Algorithm \ref{AlgoProblem1}. Then \(J^{*}\) obtained from  Algorithm \ref{AlgoProblem1} solves \eqref{keyproblem1} with \(\ell=2\).
\end{lemma}

\begin{proof}
Assume, as hypothesized that \(\mu(A, I_F)\leq 2\). By Theorem \ref{t:observabilityindex}, the digraph \(G(A, I_F)\) is spanned by a disjoint union of output cacti such that each cactus has at most \(2\) state vertices. In fact, each cactus is an output stem with at most \(2\) state vertices. Thus, the condition  \(\mu(A, I_F)\leq 2\) and the construction of \(G_u\) in Algorithm \ref{AlgoProblem1} confirms that each \(\st_i\in \mathcal{A}\setminus F\) can be associated with a distinct \(\st_j \in F\) such that \((\st_i, \st_j)\in E_{\mathcal{A}\setminus F}\).  Let \(M\) be a maximum weight matching in \(G_u\) under the weight function \(w\) defined in \eqref{e:weight}. The weight structure (Step 2) of Algorithm \ref{AlgoProblem1}  guarantees that every vertex \(\st_i \in \mathcal{A}\setminus F\) is saturated by an edge \(e=(\st_i, \st_j) \in M\) for some \(\st_j \in J^*\). The set of vertices unsaturated by \(M\) lies in \(F\) and belongs to \(J^*\) (Step 5). The output \(J^{*}\) of Algorithm \ref{AlgoProblem1} ensures that we obtain a disjoint union of output stems covering all state vertices and each stem has at most \(2\) state vertices. Thus, \(\mu(A, I_{J^*})\leq 2\). 

Let \(\abs{\mathcal{A}\setminus F}=r\).  Suppose that there exists  \(J^{\prime} \subset F\) such that \(\abs{J^{\prime}}< \abs{J^*}\) and \(\mu(A, I_{J^{\prime}})\leq 2\). Since each output is directly connected to a distinct state in \(J^{\prime}\), we can obtain a disjoint union of cacti or output stems \(P\) spanning \(G(A, I_{J^{\prime}})\) such that each output cactus has at most \(2\) state vertices with each output being directly connected to the tip of a state stem in \(P\). Note that each output cactus in \(P\) can be one of the following three types:
\begin{itemize}[label=\(\circ\), leftmargin=*]
\item An output stem \(P_i\) consisting of two state vertices connected by an directed edge of the form \((\st_a,\st_b)\), where \(\st_a \in \mathcal{A}\setminus F \) and \(\st_b\in J^{\prime}\).
\item An output stem \(P_k\) comprising of an edge of the form \((\st_c,\st_d)\), where \(\st_c\in F\) and \(\st_d \in J^{\prime}\).
\item An output stem \(P_j\) consisting of only one state vertex \(\st_e \in J^{\prime}\).
\end{itemize}

Accordingly, \(P\) is the union of \(\lbrace P_i\rbrace_{i=1}^{r}\cup \lbrace P_k\rbrace_{k=r+1}^{n}\cup \lbrace P_j \rbrace_{j=n+1}^{\abs{J^{\prime}}}\) and \(n \leq \abs{J^{\prime}}\). Here each \(P_i\) is associated with an edge \(e_i \in E_{\mathcal{A}\setminus F}\) of \(G_u\), where \(1\leq i\leq r\). Similarly, \(P_k\) is associated with an edge \(e_k\in E_F\) of \(G_u\). Observe that the collection of edges  \(\lbrace e_i\rbrace_{i=1}^{r} \cup \lbrace e_k\rbrace_{k=r+1}^{n} \) refers to a  matching \(M^{\prime}\) in the undirected graph \(G_u\). Clearly, both \(M\) and \(M^{\prime}\) saturates all the vertices in \(\mathcal{A}\setminus F\). The condition  \(\abs{J^{\prime}}< \abs{J^*}\) implies that the edges \(e_k\) of the form \((\st_c,\st_d)\), where \(\st_c, \st_d\in F\), must be greater in number in \(M^{\prime}\) than in \(M\). Since the edges in \(\lbrace e_k\rbrace_{k=r+1}^{n}\) have unit weights each \(w(M^{\prime})> w(M)\). This gives a contradiction and completes the proof.
\end{proof}
\begin{remark}
\label{r:soi2P2}
Since \((\mathcal P_1^\prime)\) is a special case of \eqref{keyproblem1}, a solution of \((\mathcal P_1^\prime)\) can be obtained by using Algorithm \ref{AlgoProblem1} by setting \(F=\mathcal{A}\). Indeed, it follows that the edge set \(E_u = E_F\) since \(E_{\mathcal{A}\setminus F}\) is empty, and since unit weight is assigned to each edge in \(E_F\), determining a maximum weight matching in Step \(3\) reduces to finding a maximum matching \(M\) in \(G_u\). Therefore, the set of state vertices \(J^*\) (obtained in Step \(5\) of Algorithm \ref{AlgoProblem1}) such that \(\mu(A, I_{J^{*}})\leq 2\), satisfies \(\abs{J^*}=\abs{M}+ (d-2\abs{M})=d-\abs{M}\), where \(M\) is a maximum matching in \(G_u\) and \((d-2\abs{M})\) is the number of state vertices unsaturated by \(M\) in \(G_u\). The proof of optimality of the obtained solution \(J^{*}\) is omitted due to similarity with the proof of Lemma \ref{l:lemmasoi2}.
\end{remark}

We shall now discuss the complexity of \((\mathcal P_1^\prime)\) and show that it is hard whenever the bound on the structural observability index \(\ell\geq 3\). The decision version of \((\mathcal P_1^\prime)\) is given by:

\noindent \textit{Instance: \(A \in \lbrace 0, \ast \rbrace^{d \times d}, C=I_{d}\in \lbrace 0, \ast\rbrace^{d \times d},\) and two positive integers \(\ell \in [d]\) and \(K\leq d\).}\\
\textit{Question: Does there exists a set \(J \subset \mathcal{A}\) with \(\abs{J}\leq K\) such that \(\mu(A, I_J)\leq \ell\).}

We use the following NP-complete problem to show the hardness of \((\mathcal P_1^\prime)\) whose decision version is given by:

\noindent\textit{Instance: An undirected graph \(G=(V, E)\), weight \(w(v)\in \mathbb{N}\) for each \(v \in V\), positive integers \(\ell\) and \(K\leq \abs{V}\). }\\
\textit{Question: Can the vertices in \(V\) be partitioned into \(k\leq K\) disjoint sets \(V_1, V_2, \ldots, V_k\) such that, for \(1 \leq i \leq k\), the induced subgraph of \(G\) by \(V_i\) is connected and the sum of the weights of the vertices in \(V_i\) does not exceed \(\ell\)?} 

The preceding problem is the \emph{Bounded Component Spanning Forest (BCSF)} problem, and it remains NP-complete even if the weights of all the vertices in \(G\) is equal to \(1\) and \(\ell\) is any fixed integer larger than \(2\) \cite{ref:GarJoh-79}. 
\begin{theorem}
\label{t:NPhardlge3}
\((\mathcal P_1^\prime)\) is NP-complete if the desired bound on the structural observability index is any fixed integer greater than or equal to \(3\).
\end{theorem}
\begin{proof}
Given \(A\) and a positive number \(K\), for any \(J\subset \mathcal{A}\), it can be verified whether \(\abs{J}\leq K\) and \(\mu(A, I_J)\leq \ell\) in polynomial time \cite{ref:SueDau-97}. Therefore, the decision version of \((\mathcal P_1^\prime)\) is in NP. 

Let \(G=(V, E)\) be an undirected graph with \(\abs{V}=r\).  We define a weight function \(w:V \to \mathbb{N}\) such that \(w(v)=1\) for all \(v \in V\). For each connected component induced by \(V_i\) in \(G\), \(\sum_{v_j \in V_i} w(v_j)=\abs{V_i}\) for \(1\leq i\leq k\), i.e., the sum of the weights of the vertices in \(V_i\) is equal to the number of vertices in the component induced by \(V_i\).  We construct a state matrix \(A \in \lbrace0, \ast\rbrace^{r\times r}\) as follows: for every vertex \(v_i \in V\), we associate a state \(x_i\) and with each undirected edge \(e_{ij}=(v_j, v_i)\) we associate two non-zero entries \(A_{ij}=\ast\) and \(A_{ji}=\ast\). We assume that \(A_{ii}=\ast\) for all \(i\in [r]\). Therefore, \(G(A)\) has only bidirectional edges, and a self-loop at every vertex. Let \(C\) be a matrix with a diagonal structure and dimension \(r\), i.e., \(C=I_r\in \lbrace 0, \ast\rbrace^{r \times r}\). Let \(\ell\) be a fixed integer larger than \(2\).

We will prove the following statement: \(G\) has a partition of at most \(K\) connected components such that the sum of the weights of the vertices in each component is at most \(\ell\) if and only if there exists a \(J \subset\mathcal{A}\) such that \(\abs{J}\leq K\) and \(\mu(A, I_J)\leq \ell\). Let us suppose that \(G\) has a partition of  at most \(K\) connected components. Let \(V_1, V_2, \ldots, V_k\) be the vertex sets of the components of that partition, where \(k\leq K\). Since \(w(v)=1\) for all \(v \in V\), \(\abs{V_i}\leq \ell\) for all \(i \in [k]\). Since each component is connected, each has an undirected spanning tree. Select one vertex \(v_i\) from each component \(V_i\) of \(G\). Let \(J= \bigcup_{i=1}^{k}\lbrace \st_i \rbrace\), where \(x_i\) is the state vertex associated with \(v_i\) selected from the component induced by \(V_i\). Observe that \(\abs{J}=k\leq K\). These vertices in \(J\) are directly connected to distinct output vertices. Thus, \(G(A, I_J)\) has a spanning forest topped at output vertices such that each directed tree has at most \(\ell\) state vertices. Hence, \(\mu(A, I_J)\leq \ell\). Conversely, suppose that there exists a \(J\subset \mathcal{A}\) such that \(\abs{J}\leq K\) and \(\mu(A, I_J)\leq \ell\). Then by Lemma \ref{l:observabilityindex}, \(G(A, I_J)\) has a spanning forest topped at output vertices such that the number of state vertices in each directed tree being at most \(\ell\) and each state vertex in \(J\) is directly connected to a distinct output. Since edges between distinct vertices in \(G(A)\) are bidirectional, each directed tree without the output vertex corresponds to a connected component of \(G\) with at most \(\ell\) state vertices. Therefore, we obtain a partition, of cardinality at most \(K\) and weight of each component at most \(\ell\), of the graph \(G\). 

The bounded component spanning forest problem remains NP-complete when the weights \(w(v)=1\) for all \(v\in V\) and the bound \(\ell\) is fixed in  \(\lbrace 3,4, \ldots\rbrace\). Under these conditions on the weights and the bound, in the above analysis, we provided a polynomial time reduction showing that the bounded component spanning forest problem has a solution of size at most \(K\) with the sum of the weights of the vertices in each component at most \(\ell\) if and only if \((\mathcal{P}_1^{\prime})\) has a solution for the constructed instance of cardinality at most \(K\) with the bound on the structural observability index being \(\ell\). Thus, it follows from the above two points that \((\mathcal{P}_1^{\prime})\) is NP-complete whenever the bound \(\ell\) on the structural observability index is fixed to a value \(\geq 3\), thereby completing the proof.
\end{proof}  
\vspace*{-0.4 cm}
Specifically, for parameters \(A\) and \(C=I_{d}\), a solution to \eqref{keyproblem1} provides, in particular, a solution to \((\mathcal P_1^\prime)\). Therefore, \eqref{keyproblem1} is at least as difficult as \((\mathcal P_1^\prime)\). Consequently, \eqref{keyproblem1} is NP-complete even if the desired bound on the structural observability index is any fixed integer larger than two.
\begin{remark}
\label{r:NP-hardp1}
	\mbox{}
	\begin{itemize}[label=\(\circ\), leftmargin=*]
		\item As discussed in \S\ref{s:related work}, NP-hardness of \((\mathcal P_1^\prime)\) is proved in \cite{ref:SerAlbGeo-17} by reducing from the Graph-Partitioning problem \cite{ref:GarJoh-79}. The proof given in \cite{ref:SerAlbGeo-17} demonstrates that an optimal solution of \((\mathcal P_1^\prime)\) leads to a feasible solution to the Graph-Partitioning problem, but optimality of the obtained solution was not addressed there. 
			In contrast, we use the decision versions of \((\mathcal{P}_1^{\prime})\) and the BCSF problem to demonstrate the equivalence between the solution of the given instance of the BCSF problem and the solution of \((\mathcal{P}_1^{\prime})\) for the constructed instance.
		\item While subclasses for which one can obtain an optimal solution of \((\mathcal{P}_1^{\prime})\) are not considered in \cite{ref:SerAlbGeo-17}, it follows from our proof of Theorem \ref{t:NPhardlge3} that \((\mathcal{P}_1^{\prime})\) is NP-hard even when \(\ell \ge 3\), the digraph \(G(A)\) has only bidirectional edges, and every state vertex has a self-loop. This is because the constructed instance of \((\mathcal{P}_1^{\prime})\) lies in this subclass. In fact, we can address this subclass via the BCSF problem because an instance of \((\mathcal{P}_1^{\prime})\) under this assumption can be reduced to an instance of the BCSF problem in polynomial time. By neglecting the self-loops, we view the digraph \(G(A)\) as a weighted undirected graph \(G_u\) by setting \(w(x_i)=1\) for all \(i\in [d]\). Given an integer \(\ell\), the output of the BCSF problem on the instance \((G_u, \ell)\) finds a decomposition of \(G_u\) with the minimum number of connected partitions such that each partition has at most \(\ell\) vertices. Then it is easy to see from our proof of Theorem \ref{t:NPhardlge3} that:
			\begin{enumerate}[leftmargin=*, label=\textup{(\roman*)}, align=left, widest=ii]
				\item a minimal partition given by the BCSF problem provides an optimal solution of \((\mathcal{P}_1^{\prime})\), and
				\item for any \(\alpha\geq 1\), an \(\alpha\)-optimal solution of the BCSF problem on the instance \((G_u, \ell)\) gives an \(\alpha\)-optimal solution of \((\mathcal{P}_1^{\prime})\) for the given instance.\footnote{Recall that an \(\alpha\)-optimal solution is a feasible solution whose value is at most \(\alpha\) times the optimal value.}
			\end{enumerate}
			Thus, under the above assumption on the digraph \(G(A)\), any approximation for the BCSF problem gives an approximation for \((\mathcal{P}_1^{\prime})\)  with the same approximation ratio. The proof technique employed here, therefore, sheds far more light into the problem than the one in \cite{ref:SerAlbGeo-17}.
	\end{itemize}
\end{remark}
Next we impose the following assumption on the digraph \(G(A)\) associated with the state matrix \(A\).

\begin{assumption}
\label{a:directedrootedtree}
We stipulate that the digraph \(G(A)=(\mathcal{A}, E_A)\) is a directed tree towards an \(x\in \mathcal{A}\) with a self-loop at every state vertex.
\end{assumption}
We have the following Lemma that plays a central role to solve \eqref{keyproblem1} when Assumption \ref{a:directedrootedtree} holds.
\begin{lemma}
\label{l:subtree}
Let \(A\in \lbrace 0, \ast \rbrace^{d\times d}\) and \(C=I_F\in \lbrace 0, \ast \rbrace^{\abs{F}\times d}\). Suppose Assumption \ref{a:directedrootedtree} holds. Let \(\ell\in [d]\) be the desired bound on the structural observability index with \(\mu(A, I_F)\leq \ell\). Let \(P=\lbrace \mathcal{T}_1, \mathcal{T}_2,\ldots, \mathcal{T}_{\abs{J^*}}\rbrace\) be a partition of \(G(A)=(\mathcal{A}, E_A)\) into the minimum number of subtrees, where each \(\mathcal{T}_i\) represents a subtree such that its tip is a vertex present in \(F\subset \mathcal{A}\) and the number of state vertices in no subtree exceeds \(\ell\). The collection \(J^{*}\subset F\) of tips of each subtree solves \eqref{keyproblem1}. \footnote{A \emph{tip} of a tree \(\mathcal{T}\) is the vertex that does not have any outgoing edges to any other vertex in \(\mathcal{T}\) and has a directed path from every vertex to it in \(\mathcal{T}\). A \emph{subtree} is a subgraph of the directed tree which satisfies all the properties of tree.} 
\end{lemma}  
\begin{proof}
Since \(\mu(A, I_F)\leq \ell\), by Lemma \ref{l:observabilityindex}, there exists a partition or a spanning forest topped at output vertices for \(G(A, I_F)\) such that each directed tree is towards an output vertex and has at most \(\ell\) state vertices. Consider the partition \(P\) of \(G(A)\) such that each subtree \(\mathcal{T}_i=(X_i, E_{X_i})\) has at most \(\ell\) state vertices, where \(X_i \subset \mathcal{A}\), \(E_{X_i}\) is the edges in the subtree induced by \(X_i\) for \(1\leq i\leq \abs{J^*}\) (excluding self-loops). Let \(\st_i \in F\) be the tip of the subtree \(\mathcal{T}_i\). It follows that \(G_i=(X_i \sqcup \lbrace i\rbrace, E_{X_i} \sqcup \lbrace \st_i, i\rbrace)\), where \(\st_i \in J^{*}\) and \(i \in \mathcal{C}\) (output set), is a directed tree towards \(i\in \mathcal{C}\). \(J^{*}\) contains the tip of each subtree. Therefore, the digraph \(G(A, I_{J^*})\) has a spanning forest with a collection of directed trees \(\lbrace G_i\rbrace_{i=1}^{\abs{J^{*}}}\) with at most \(\ell\) state vertices in each \(G_i\). Since every state vertex has a self-loop, by Lemma \ref{l:observabilityindex} \(\mu(A, I_{J^*})\leq \ell\). The minimality assertion follows by the fact that if \(J^{\prime}\subset F\) has size less than \(J^*\), then there exists another partition of the tree into subtrees of smaller cardinality than \(P\) where the tip of each subtree is a vertex lying in \(F\). This leads to a contradiction and confirms that \(J^{*}\) solves \eqref{keyproblem1}. 
\end{proof}
In the further analysis to solve \eqref{keyproblem1}, when Assumption \ref{a:directedrootedtree} holds, we neglect the self-loops present at every state vertex of \(G(A)\). We provide the merging procedure that will be used later in  Algorithm \ref{algorithmdirectedtree}. Given \(F \subset \mathcal{A}\), we use this procedure in Algorithm \ref{merging} to transform the given directed tree \(G(A)\) into another directed tree \(G^{\prime}(A)\) whose vertex set is \(F\). Given the directed tree \(G(A)=(\mathcal{A}, E_A)\) and \(F \subset \mathcal{A}\), define a variable \(h\) associated with every \(\st_i \in \mathcal{A}\setminus F\) such that \(h(\st_i)=\st_j\) if \((x_i, x_j)\in E_A\) and \(\st_i \neq \st_j\) (i.e., excluding self-loops).
\begin{algorithm}
\KwIn{directed tree \(G(A)=(\mathcal{A}, E_A)\), \(F \subset\mathcal{A}\), variable \(h\), and weight \(w(\st_i)=1\) for all \(\st_i\in \mathcal{A}\)}
	\KwOut{directed tree \(G^{\prime}(A)\) obtained from \(G(A)\) with vertex set as \(F\) and weight \(w: F \to \N \)}
	
	\nl If \(F=\mathcal{A}\)
	
	\quad return \(G(A)\) and \(w(\st_i)=1\) for all \(\st_i\in F\)
	
	\nl else
	
	\nl for \(k=\) maximum level in \(G(A)\) down to 1 \textbf{do}
	
	\nl begin process \(k{th}\) level
	
	\nl \textbf{while} there exists \(\st_i\in \mathcal{A}\setminus F\) in level \(k\) \textbf{do}
	
	\nl if \(\st_j=h(x_i)\) then
	
	\quad set \(w(\st_j)=w(\st_j)+w(\st_i)\) and merge \(\st_i\) to \(\st_j\)
	
	\nl end
	
	\nl end \textbf{while}
	
	\nl end process level
	
	\nl end for
	
	\nl end 
	
	\caption{Merging procedure}
\label{merging}
\end{algorithm}

In  Algorithm \ref{merging}, we begin with the maximum level (i.e., from the leaves) and go down to the tip of the tree (level 1). Note that when a state vertex \(\st_i\in \mathcal{A}\setminus F\) is merged with \(\st_j\) such that \(h(\st_i)=\st_j\) in Step 6, then the edges of the given tree incident to \(\st_i\) are now incident to \(\st_j\), and the digraph gets modified. Since the maximum number of levels is bounded above by \(d\) and the procedure comprises of only `for' and `while' loops, the overall complexity of Algorithm \ref{merging} is \(O(d)\).

By Lemma \ref{l:subtree} it is clear that to solve \eqref{keyproblem1} we need to find a minimal partition of \(G(A)\) into subtrees such that the tip of each subtree lies in \(F\) and the number of state vertices in each subtree does not exceed the bound \(\ell\) imposed on the structural observability index. For this purpose, we use the problem of finding a minimal partition \(P\) of a tree with positive weight assigned to every vertex in the directed tree such that the sum of the weights of the vertices of  no subtree exceed  a prespecified value \(\ell\); This problem is solved in \cite{ref:KunMis-77}. The key idea of the algorithm in \cite{ref:KunMis-77} is discussed briefly next. Let \(\mathcal{T}\) be the given tree and \(\mathcal{T}_{m}\) be a subtree with tip at \(\st_m\), \(S(m)\) be the set of children of \(\st_m\), i.e., having a directed edge to \(\st_m\), \(w(m)\) be the weight of the vertex \(\st_m\), and \(W(m)\) be the sum of the weights of the vertices in the subtree \(\mathcal{T}_{m}\). Each vertex has weight at most \(\ell\). Given the pre-specified bound \(\ell \in \N\), if \(\st_m\) is a vertex such that \(W(m)> \ell\) and \(W(k)\leq \ell\) for all \(\st_k\in S(m)\), then the edge \((\st_{k_0}, \st_m)\) is removed where \(W(k_0)=\max_{\st_k\in S(m)} W(k)\). This results in the removal of the subtree whose tip is \(x_{k_0}\), \(\mathcal{T}_{k_0}\), from the tree \(\mathcal{T}\) and \(\mathcal{T}_{k_0}\) becomes a component of the partition \(P\). By proceeding along the tree level by level, starting from the leaves and ending at the tip, we locate the vertex \(\st_m\) in  this procedure. The resultant tree, obtained after deletion of the subtree whose tip is \(\st_{k_0}\), is analysed further by employing the same procedure as given above. At any stage of the algorithm, a single tree is modified by the deletion of a subtree. We employ this linear time algorithm provided in \cite{ref:KunMis-77} to find a solution of \eqref{keyproblem1} via Algorithm \ref{algorithmdirectedtree}.

\begin{algorithm}
\KwIn{\(A \in \lbrace 0, \ast \rbrace^{d \times d}\) and \(F \subset\mathcal{A}\) such that \(\mu(A, I_F)\leq \ell\)}
	\KwOut{A solution of \eqref{keyproblem1},  \(J^{*}\subset F\).}
	
	\nl Assign weight \(w(\st_i)=1\) for all \(\st_i \in \mathcal{A}\) 
	
	\nl Use merging procedure given in Algorithm \ref{merging} to obtain \(G^{\prime}(A)\) and weight function \(w:F \to \N \).
	
	\nl Use Algorithm of \cite{ref:KunMis-77} on \(G^{\prime}(A)\) with weight function \(w\) and bound \(\ell\) to find an optimal partition \(P\)
	
	\nl \(J^{*}\) is the collection of tip of each subtree of \(P\)
	
\caption{Solve Problem \eqref{keyproblem1} under Assumption \ref{a:directedrootedtree}}
\label{algorithmdirectedtree}
\end{algorithm}

Since each step in Algorithm \ref{algorithmdirectedtree} has linear complexity, we obtain a solution to \eqref{keyproblem1} in linear time. It is easy to see that we can obtain a minimal partition of \(G(A)\) which satisfies the condition that each subtree has at most \(\ell\) number of state vertices and its tip lies in \(F\) from the minimal partition of \(G^{\prime}(A)\) obtained by using \cite{ref:KunMis-77} in Step 3 of Algorithm \ref{algorithmdirectedtree}.

Next, we demonstrate the steps in Algorithm \ref{algorithmdirectedtree} to identify a solution of \eqref{keyproblem1}  via an example given below from Fig.~ \ref{f:digraphsol} to Fig.~ \ref{f:figsolution}. 

\begin{figure}[h]
\centering
   \begin{tikzpicture}
\node[draw=black,fill=orange!30] (1) at (0,0) [] {$\st_1$};
\node[draw=black,fill=orange!30] (2) at (2,-1.5) [] {$\st_2$};
\node[draw=black,fill=orange!30] (3) at (-2,-1.5) [] {$ \st_3 $};
\node  (4) at (0,-1.5) [] {$ \st_4 $};
\node (5) at (-3.5,-3) [] {$\st_5$};
\node (6) at (-1.5,-3) [] {$\st_6$};
\node[draw=black,fill=orange!30] (7) at (0,-3) [] {$ \st_7 $};
\node (8) at (3,-3) [] {$ \st_8 $};
\node[draw=black,fill=orange!30] (9) at (-4.5,-4.5) [] {$\st_9$};
\node (10) at (-2.8,-4.5) [] {$\st_{10}$};
\node (11) at (-2.2,-4.5) [] {$ \st_{11} $};
\node[draw=black,fill=orange!30] (12) at (-1.5,-4.5) [] {$ \st_{12} $};
\node (13) at (-0.8,-4.5) [] {$\st_{13}$};
\node (14) at (0,-4.5) [] {$\st_{14}$};
\node (15) at (1.0,-4.5) [] {$ \st_{15} $};
\node (16) at (4,-4.5) [] {$ \st_{16} $};
\node[draw=black,fill=orange!30] (17) at (3,-4.5) [] {$\st_{17}$};
\node (18) at (2,-4.5) [] {$\st_{18}$};
    \path [->,solid, line width=1.0pt](2) edge node[right] {} (1);
    \path [->,solid, line width=1.0pt](3) edge node[right] {} (1);
    \path [->,solid, line width=1.0pt](4) edge node[right] {} (1);
    \path [->,solid, line width=1.0pt](5) edge node[right] {} (3);
    \path [->,solid, line width=1.0pt](8) edge node[right] {} (2);
    \path [->,solid, line width=1.0pt](18) edge node[right] {} (8);
    \path [->,solid, line width=1.0pt](17) edge node[right] {} (8);
    \path [->,solid, line width=1.0pt](16) edge node[right] {} (8);
    \path [->,solid, line width=1.0pt](15) edge node[right] {} (7);
    \path [->,solid, line width=1.0pt](14) edge node[right] {} (7);
    \path [->,solid, line width=1.0pt](6) edge node[right] {} (3);
    \path [->,solid, line width=1.0pt](9) edge node[right] {} (5);
    \path [->,solid, line width=1.0pt](12) edge node[right] {} (6);
    \path [->,solid, line width=1.0pt](10) edge node[right] {} (5);
    \path [->,solid, line width=1.0pt](11) edge node[right] {} (6);
    
    \path [->,solid, line width=1.0pt](13) edge node[right] {} (6);
    \path [->,solid, line width=1.0pt](7) edge node[right] {} (4);
\end{tikzpicture}
 \caption{Illustration of a directed tree \(G(A)=(\mathcal{A}, E_A)\) towards \(\st_1\) associated with \(A\). Each state vertex has a self-loop, but we do not depict it to avoid clutter. The orange coloured vertices represent the state vertices present in \(F\). It is assumed that \(w(\st_i)=1\) for all \(\st_i \in \mathcal{A}\). Clearly, \(\st_1\in F\) to ensure structural observability. Notice that \(F\neq \mathcal{A}\). Let the desired bound on the structural observability index is \(\ell=7\).}
\label{f:digraphsol}

\end{figure}
\begin{figure}[h]
\centering
   \begin{tikzpicture}
\node[draw=black,fill=orange!30] (1) at (0,0) [] {$\st_1$};
\node[draw=black,fill=orange!30] (2) at (2,-1.5) [] {$\st_2$};
\node[draw=black,fill=orange!30] (3) at (-2,-1.5) [] {$ \st_3 $};
\node  (4) at (0,-1.5) [] {$ \st_4 $};
\node[draw=black,fill=green!10]  (5) at (-3.5,-3) [] {$\st_5$};
\node[draw=black,fill=green!10]  (6) at (-1.5,-3) [] {$\st_6$};
\node[draw=black,fill=orange!30] (7) at (0,-3) [] {$ \st_7 $};
\node[draw=black,fill=green!10]  (8) at (3,-3) [] {$ \st_8 $};
\node[draw=black,fill=orange!30] (9) at (-3.5,-4.5) [] {$\st_9$};
\node[draw=black,fill=orange!30] (12) at (-1.5,-4.5) [] {$ \st_{12} $};
\node[draw=black,fill=orange!30] (17) at (3,-4.5) [] {$\st_{17}$};
    \path [->,solid, line width=1.0pt](2) edge node[right] {} (1);
    \path [->,solid, line width=1.0pt](3) edge node[right] {} (1);
    \path [->,solid, line width=1.0pt](4) edge node[right] {} (1);
    \path [->,solid, line width=1.0pt](5) edge node[right] {} (3);
    \path [->,solid, line width=1.0pt](8) edge node[right] {} (2);
    \path [->,solid, line width=1.0pt](17) edge node[right] {} (8);
    \path [->,solid, line width=1.0pt](6) edge node[right] {} (3);
    \path [->,solid, line width=1.0pt](9) edge node[right] {} (5);
    \path [->,solid, line width=1.0pt](12) edge node[right] {} (6);
    
    \path [->,solid, line width=1.0pt](7) edge node[right] {} (4);
\end{tikzpicture}
\caption{Illustration of the merging procedure: Note that \(h(\st_{10})=\st_5\), \(h(\st_{11})=\st_6\), \(h(\st_{13})=\st_6\), \(h(\st_{14})=\st_7\), \(h(\st_{15})=\st_7\), \(h(\st_{16})=\st_8\), and \(h(\st_{18})=\st_8\) and number of levels is four. \(\st_{10}\) is merged with \(\st_5\), \(\st_{11}\) and \(\st_{13}\) are merged with \(\st_6\), \(\st_{14}\) and \(\st_{15}\) are merged to \(\st_7\), and \(\st_{16}\) and \(\st_{18}\) are merged with \(\st_8\). Thus, \(w(\st_5)=2, w(\st_6)=3, w(\st_7)=3\), and \(w(\st_8)=3\). The rest of the vertices have unit weights.}
\end{figure}

\begin{figure}[h]
\begin{center}
   \begin{tikzpicture}
\node[draw=black,fill=orange!30] (1) at (0,0) [] {$\st_1$};
\node[draw=black,fill=orange!30] (2) at (2,-1.5) [] {$\st_2$};
\node[draw=black,fill=orange!30] (3) at (-2,-1.5) [] {$ \st_3 $};
\node[draw=black,fill=orange!30] (7) at (0,-1.5) [] {$ \st_7 $};
\node[draw=black,fill=orange!30] (9) at (-3,-3) [] {$\st_9$};
\node[draw=black,fill=orange!30] (12) at (-1.5,-3) [] {$ \st_{12} $};
\node[draw=black,fill=orange!30] (17) at (2,-3) [] {$\st_{17}$};
    \path [->,solid, line width=1.0pt](2) edge node[right] {} (1);
    \path [->,solid, line width=1.0pt](3) edge node[right] {} (1);
    \path [->,solid, line width=1.0pt](4) edge node[right] {} (1);
    \path [->,solid, line width=1.0pt](9) edge node[right] {} (3);
    \path [->,solid, line width=1.0pt](17) edge node[right] {} (2);
    \path [->,solid, line width=1.0pt](12) edge node[right] {} (3);
    
    \path [->,solid, line width=1.0pt](7) edge node[right] {} (1);
\end{tikzpicture}
\end{center}
\caption{Note that \(h(\st_5)=\st_3\), \(h(\st_6)=\st_3\), \(h(\st_8)=\st_2\), \(h(\st_4)=\st_1\). \(\st_{5}\) and \(\st_{6}\) are merged with \(\st_3\), \(\st_{4}\) is merged with \(\st_1\), and \(\st_8\) is merged with \(\st_2\). Thus \(w(\st_1)=2, w(\st_2)=4\), \(w(\st_3)=6\), and \(w(\st_7)=3\). The rest of the vertices have unit weights. Note that we obtain \(G^{\prime}(A)\) from \(G(A)\) with all the vertices belonging to \(F\) and modified weight function \(w:F\to \mathbb{Z}^+\).}
\label{f:illustredugraph}
\end{figure}
We demonstrate in Fig.~\ref{f:figillustP} an optimal partition \(P\) corresponding to the graph \(G^{\prime}(A)\) depicted in Fig.~\ref{f:illustredugraph}. Fig.~\ref{f:figsolution} depicts the solution of \eqref{keyproblem1} for the digraph \(G(A)\) shown in Fig.~\ref{f:digraphsol}.
\begin{figure}[h!]
\begin{center}
   \begin{tikzpicture}
\node[draw=black,fill=orange!30] (1) at (0,0) [] {$\st_1$};
\node[draw=black,fill=orange!30] (2) at (2,-1.5) [] {$\st_2$};
\node[draw=black,fill=orange!30] (3) at (-2,-1.5) [] {$ \st_3 $};
\node[draw=black,fill=orange!30] (7) at (0,-1.5) [] {$ \st_7 $};
\node[draw=black,fill=orange!30] (9) at (-3,-3) [] {$\st_9$};
\node[draw=black,fill=orange!30] (12) at (-1.5,-3) [] {$ \st_{12} $};
\node[draw=black,fill=orange!30] (17) at (2,-3) [] {$\st_{17}$};
    \path [->,solid, line width=1.0pt](4) edge node[right] {} (1);
    \path [->,solid, line width=1.0pt](17) edge node[right] {} (2);
    \path [->,solid, line width=1.0pt](12) edge node[right] {} (3);
    
    \path [->,solid, line width=1.0pt](7) edge node[right] {} (1);
\end{tikzpicture}
 \end{center}
\caption{Illustration of an optimal partition \(P\) of \(G^{\prime}(A)\) when the bound on the sum of vertex-weights of each subtree is \(\ell=7\) which is equal to the desired bound on the structural observability index, obtained by Step 3 in Algorithm \ref{algorithmdirectedtree}.}
\label{f:figillustP}
\end{figure}
\begin{figure}[h]
\begin{center}
   \begin{tikzpicture}
\node[draw=black,fill=orange!30] (1) at (0,0) [] {$\st_1$};
\node[draw=black,fill=orange!30] (2) at (2,-1.5) [] {$\st_2$};
\node[draw=black,fill=orange!30] (3) at (-2,-1.5) [] {$ \st_3 $};
\node  (4) at (0,-1.5) [] {$ \st_4 $};
\node (5) at (-3.5,-3) [] {$\st_5$};
\node (6) at (-1.5,-3) [] {$\st_6$};
\node[draw=black,fill=orange!30] (7) at (0,-3) [] {$ \st_7 $};
\node (8) at (3,-3) [] {$ \st_8 $};
\node[draw=black,fill=orange!30] (9) at (-4.5,-4.5) [] {$\st_9$};
\node (10) at (-2.8,-4.5) [] {$\st_{10}$};
\node (11) at (-2.2,-4.5) [] {$ \st_{11} $};
\node[draw=black,fill=orange!30] (12) at (-1.5,-4.5) [] {$ \st_{12} $};
\node (13) at (-0.8,-4.5) [] {$\st_{13}$};
\node (14) at (0,-4.5) [] {$\st_{14}$};
\node (15) at (1.0,-4.5) [] {$ \st_{15} $};
\node (16) at (4,-4.5) [] {$ \st_{16} $};
\node[draw=black,fill=orange!30] (17) at (3,-4.5) [] {$\st_{17}$};
\node (18) at (2,-4.5) [] {$\st_{18}$};
    \path [->,solid, line width=1.0pt](4) edge node[right] {} (1);
    \path [->,solid, line width=1.0pt](5) edge node[right] {} (3);
    \path [->,solid, line width=1.0pt](8) edge node[right] {} (2);
    \path [->,solid, line width=1.0pt](18) edge node[right] {} (8);
    \path [->,solid, line width=1.0pt](17) edge node[right] {} (8);
    \path [->,solid, line width=1.0pt](16) edge node[right] {} (8);
    \path [->,solid, line width=1.0pt](15) edge node[right] {} (7);
    \path [->,solid, line width=1.0pt](14) edge node[right] {} (7);
    \path [->,solid, line width=1.0pt](6) edge node[right] {} (3);
    \path [->,solid, line width=1.0pt](12) edge node[right] {} (6);
    \path [->,solid, line width=1.0pt](10) edge node[right] {} (5);
    \path [->,solid, line width=1.0pt](11) edge node[right] {} (6);
    
    \path [->,solid, line width=1.0pt](13) edge node[right] {} (6);
    \path [->,solid, line width=1.0pt](7) edge node[right] {} (4);
\end{tikzpicture}
 \end{center}
\caption{Illustration of the partition \(P\) in the digraph \(G(A)\). Here \(J^*=\lbrace \st_1, \st_2, \st_3, \st_9\rbrace \subset F\) and solves \eqref{keyproblem1} for \(\ell=7\).}
\label{f:figsolution}
\end{figure}

\section{Cardinality constrained sensor placement}
\label{s:cardinality constrained sensor placement}
Suppose \(A\in \lbrace 0, \ast \rbrace^{d\times d}\) and \(C\in \lbrace  0, \ast \rbrace^{p\times d}\) are given. Recall \(\mathcal{A}=\lbrace \st_1, \st_2, \ldots, \st_d\rbrace\) and \(\mathcal{C}=\lbrace 1, 2, \ldots, p\rbrace\) are the sets of state and output vertices of the digraph \(G(A,C)\) respectively. Without loss of generality, we assume that each output measures at least one state. For \(S \subset \mathcal{C}\), \(C(S)\) is a submatrix obtained by retaining the rows corresponding to the output vertices in \(S\). A set of state vertices \(\mathcal{A}^{\prime} \subset \mathcal{A}\) in \(G(A)\) is said to be \emph{structurally observable by \(S\)} if the induced subgraph \(G(A^{\prime}, C^{\prime}(S))=(\mathcal{A}^{\prime}\sqcup S, E_{A^{\prime}}\sqcup E^{\prime}_{S})\) of \(G(A,C)\) by \(\mathcal{A}^{\prime}\sqcup S\) satisfies both the conditions given in Theorem \ref{t:observability}(b), where \(E_{A^{\prime}}\) contains the edges between the state vertices in \(\mathcal{A}^{\prime}\) and \(E^{\prime}_{S}\subset E_C\) contains only the edges between \(\mathcal{A}^{\prime}\) and \(S\) and \(A^{\prime}\in \lbrace 0, \ast \rbrace^{\abs{\mathcal{A}^{\prime}}\times \abs{{\mathcal{A}^{\prime}}}}\), \(C^{\prime}(S)\in \lbrace 0, \ast\rbrace^{\abs{S}\times \abs{\mathcal{A}^{\prime}}}\). In other words, \((A^{\prime}, C^{\prime}(S))\) is structurally observable. 

We address the problem of selecting a set \(S\) of output vertices of cardinality at most \(r\) \((1\leq r \leq p)\) such that maximum number of state vertices are structurally observable by \(S\) in the digraph \(G(A)\). For \(S\subset \mathcal{C}\), we define \(\subfun: 2^{\mathcal{C}} \to \Real\) by
\begin{multline}
\label{e:definec(S)1}
\subfun(S)\Let \max\Big\lbrace \abs{\mathcal{A^{\prime}}}\;\Big|\; \mathcal{A^{\prime}}\subset \mathcal{A}\; \text{and}\; (A^{\prime}, C^{\prime}(S))\; \text{is structurally observable}\Big\rbrace.
\end{multline}
In simple words, \(\subfun(S)\) is the size of the largest subgraph of \(G(A)\) that is structurally observable by \(S\). The value of \(\frac{\subfun(S)}{d}\), clearly, lies in the interval \([0,1]\); the fraction \(\frac{\subfun(S)}{d}\) takes the value \(1\) if the set of all state variables of \(G(A)\) is structurally observable by \(S\), and \(0\) if no state variable is observable by \(S\). Based on \eqref{e:definec(S)1}, given \(A\in \lbrace 0, \ast \rbrace^{d\times d}\) and \(C\in \lbrace  0, \ast \rbrace^{p\times d}\), we have the problem of selecting up to \(r\) (\(1\leq r\leq p\)) output vertices so as maximize the following function:

\begin{equation}
\begin{aligned}
\label{e:submodularC(S)output}
&\maximize_{S\subset \mathcal{C}} && \frac{\subfun(S)}{d}\\ 
& \sbjto && \abs{S}\leq r.
\end{aligned}
\tag{\(\mathcal P_2\)}
\end{equation}
 
Next, we prove that \eqref{e:submodularC(S)output} is NP-hard by reducing to \eqref{e:submodularC(S)output} from the \emph{maximum cover problem}, the latter being a well-known NP-hard problem \cite{ref:Hoc-96}.
\begin{theorem}
\label{t:p3NPhard}
\eqref{e:submodularC(S)output} is NP-hard.
\end{theorem}
\begin{proof}
Consider the maximum cover problem: Given a positive integer \(r\) and a collection of sets in \(K=\lbrace Z_1, Z_2, \ldots, Z_p\rbrace\) such that each \(Z_i\) contains some elements, find a subset \(\hat{K}\subset K\) of sets such that \(\abs{\hat{K}}\leq r\) and the number of covered elements \(\abs{\bigcup_{\substack{Z_i\in \hat{K}}}Z_i}\) is maximized. Let \(\bigcup_{\substack{Z_i\in K}} Z_i=\lbrace s_j\rbrace_{j=1}^d\).

To prove the NP-hardness, we build an instance of \eqref{e:submodularC(S)output} starting from an instance of the maximum cover problem in polynomial time. Each element \(s_j\) is associated with a state vertex \(x_j\) and each set \(Z_i\) is associated to an output vertex \(i\) giving  the output set \(\mathcal{C}=\lbrace 1,2,\ldots, p\rbrace\). The corresponding state and output matrices are: \(A=I_d\in \lbrace 0, \ast \rbrace^{d\times d}\) as defined earlier in \S\ref{s:Minimal sensor placement problem} and \(C\in \lbrace 0, \ast\rbrace^{p\times d}\) with \(C_{ij}=\ast \;\text{if}\; s_j\in Z_i\), and \(0\) otherwise, for \(i\in[p]\), \(j\in[d]\). Each state vertex has a self-loop in \(G(A, C)\). 

For a given integer \(m\leq d\), we prove that there is a set \(\hat{S}\subset \mathcal{C}\) of size at most \(r\) such that \(\Xi(\hat{S})\geq m\) if and only if there exists a \(\hat{K} \subset K\) of cardinality at most \(r\) such that \(\abs{\bigcup_{\substack{Z_i\in \hat{K}}}Z_i}\geq m\). Let \(\hat{S}\subset \mathcal{C}\) with \(\abs{\hat{S}}\leq r\) and  \(\Xi(\hat{S})\geq m\). We will show that \(\hat{K}=\big\lbrace Z_i\;\big |\; i\in \hat{S} \big\rbrace\) is a feasible solution of the maximum cover problem with \(\abs{\bigcup_{\substack{Z_i\in \hat{K}}}Z_i}\geq m\). Clearly, \(\abs{\hat{K}}\leq r\) as \(\abs{\hat{S}}\leq r\). Since every state vertex has a self-loop in \(G(A)\), \(\Xi({\hat{S}})\) is equal to the number of state vertices having a directed path (which in this case is an edge) to at least one of the output vertices in \(\hat{S}\). Hence, by construction, it is easy to see that \(\Xi(\hat{S})= \abs{\bigcup_{\substack{Z_i\in \hat{K}}}Z_i}\). Thus, \(\Xi(\hat{S})\geq m\) implies that \(\abs{\bigcup_{\substack{Z_i\in \hat{K}}}Z_i}\geq m\). Conversely, if \(\hat{K}\) is such that \(\abs{\hat{K}}\leq r\) and \(\abs{\bigcup_{\substack{Z_i\in \hat{K}}}Z_i}\geq m\) then \(\hat{S}=\big\lbrace i\;\big |\; Z_i\in \hat{K} \big\rbrace\) satisfies \(\abs{\hat{S}}\leq r\) and  \(\Xi(\hat{S})= \abs{\bigcup_{\substack{Z_i\in \hat{K}}}Z_i}\geq m\), and thereby completes the proof.
 \end{proof}
In fact, Theorem \ref{t:p3NPhard} implies that \eqref{e:submodularC(S)output} is NP-hard even when every state vertex has a self-loop in \(G(A)\) since the constructed instance in the proof of Theorem \ref{t:p3NPhard} belongs to this subclass. This motivates us to impose the following condition on the digraph \(G(A)\).
\begin{assumption}
\label{a:selfloop}
It is assumed that each state vertex has a self-loop in the graph \(G(A)\). Consequently, the resultant \(G(A, C)\) has no contraction.
\end{assumption}
 Given \(S\subset \mathcal{C}\), we say that a state vertex \(\st_i \in \mathcal{A}\) is \emph{accessible} by \(S\) if there exists a directed path from it to at least one of the output vertices in \(S\). Under Assumption \ref{a:selfloop}, \(\subfun(S)\) is equal to the number of state vertices accessible by \(S\) and leads to the following lemma that proves the submodularity of \(\subfun(\cdot)\).
\begin{lemma}
\label{l:submodularc(S)}
Consider the system \eqref{e:linsys} and the associated sets \(\mathcal{A}\) and \(\mathcal{C}\) as defined in \S\ref{s:Preliminaries}. Suppose Assumption \ref{a:selfloop} holds. Then \(\subfun:2^{\mathcal{C}} \to \Real\) defined in \eqref{e:definec(S)1} is a monotone non-decreasing submodular function.
\end{lemma}
\begin{proof}
By the definition of \(\subfun(\cdot)\) it is clear that for \(S \subset T\), \(\subfun(S)\leq \subfun(T)\) when Assumption \ref{a:selfloop} holds. Therefore, \(\subfun(\cdot)\) is a monotone non-decreasing  function.

Let \(S\subset T\subset \mathcal{C}\), and suppose that \(v\notin T\).  To show that \(\subfun(\cdot)\) is submodular it is enough to establish that \(\subfun(S \cup \lbrace v \rbrace)-\subfun(S)\geq \subfun(T \cup \lbrace v \rbrace)-\subfun(T)\).

Let \(m\) be the number of state vertices accessible by both \(S\) and \(v\). Similarly, let \(n\) be the number of state vertices accessible by both \(T\) and \(v\). We have
\begin{equation}
\label{e:sub1}
\begin{aligned}
\subfun(S \cup \lbrace v \rbrace)&= \subfun(S) + \subfun(v)- m,\\
\subfun(T \cup \lbrace v \rbrace)&= \subfun(T) + \subfun(v)-n.
\end{aligned}
\end{equation}

Since \(S \subset T\), we have \(m \leq n\). Consequently,
\begin{equation}
\label{e:sub2}
\subfun(v)-m \geq \subfun(v)-n.
\end{equation} 
Using \eqref{e:sub1} and \eqref{e:sub2} it follows that \(\subfun(\cdot)\) is  submodular. \qed
\end{proof}
Although maximizing a submodular function \(\subfun(\cdot)\) is an NP-hard problem, there exists efficient greedy algorithms for providing an approximate solution of \eqref{e:submodularC(S)output} \cite{ref:NemWolFis-78}, and we employ one such mechanism in Algorithm \ref{Algoobservabilitysubodularoutput} that follows: 
\begin{algorithm}[h]
	\KwIn{\(G(A, C)\) and maximum number of outputs \(r\)}
\KwOut{A set \(S^*\subset \mathcal{C}\)}
 \nl Initialization: \(S^*=\emptyset\), \(i \leftarrow 0\)
 
 \nl \textbf{while} \(i < r\) \textbf{do} 
 
 \nl  \(s^* \leftarrow \argmax_{s\in \mathcal{C}\setminus S^{*}} \subfun(S^*\cup \lbrace s \rbrace)-\subfun(S^*)\)
 
 \nl \(S^*\leftarrow S^* \cup \lbrace s^* \rbrace\)
 
 \nl \(i\leftarrow i+1\)
 
 \nl if \(\subfun(S^{*}) = d\)
 
 \quad stop
 
 \nl else
 
 \quad go to step \(2\)
 
 \nl end
 
 \nl \textbf{end while}  
 
\nl return \(S^*\); exit
 
\caption{Approximation algorithm for solving \eqref{e:submodularC(S)output}}
\label{Algoobservabilitysubodularoutput}
\end{algorithm}

The next theorem gives a qualitative estimate of how close a solution obtained by Algorithm \ref{Algoobservabilitysubodularoutput}, is from an optimal solution of \eqref{e:submodularC(S)output}:
\begin{theorem}
\label{t:maximise c(S)}
Let \(\hat{S}\) be an optimal solution of \eqref{e:submodularC(S)output}, and let \(S^*\) be a set returned by Algorithm \ref{Algoobservabilitysubodularoutput}. Then
\begin{equation}
\frac{\subfun(S^*)}{d}\geq \bigg(1-\frac{1}{e}\bigg)\frac{\subfun(\hat{S})}{d}.
\end{equation}
\end{theorem} 
\begin{proof}
\cite[Chapter III, Section 3.9, Theorem 9.3]{ref:NemWol-99} asserts that for any  monotone non-decreasing submodular function \(\subfun(\cdot)\), the greedy Algorithm \ref{Algoobservabilitysubodularoutput} returns a set satisfying \(\subfun(S^*)\geq (1-1/e )\subfun(\hat{S})\), where \(\subfun(\hat{S}) \Let \max\Big\lbrace \subfun(S)\,\Big{|}\, \abs{S}\leq r\Big\rbrace\). The assertion follows. 
\end{proof}
 Algorithm \ref{Algoobservabilitysubodularoutput} is a greedy method that progressively picks an output vertex from \(\mathcal{C}\) and adds it to \(S^{*}\) so that the maximal increase of \(\subfun(\cdot)\) is obtained at each iteration. Observe that at most \(r\) indices may be added to \(S^{*}\) (See Step 2). In Step 3 the algorithm loops over at most \(p\) indices and checks their possible contributions to increase \(\subfun(\cdot)\) when an element of \(\mathcal{C}\) is added to \(S^{*}\). Therefore, any operation in the algorithm will be at most \(rp\) times. Given \(S^{*}\), we also need to compute \(\subfun(S^{*})\) in Step 3. Under Assumption \ref{a:selfloop}, this is equivalent to computing the total number of state vertices accessible by \(S^{*} \subset \mathcal{C}\). This set of state vertices is obtained by employing depth-first search which has \(O(d^2+dr)\) complexity \cite{ref:CorLeiRivSte-09}. Therefore, the overall complexity of Algorithm \ref{Algoobservabilitysubodularoutput} is \(O(rp(d^2+dr))\). In practice, greedy-type Algorithm \ref{Algoobservabilitysubodularoutput} for submodular maximization often outperform their worst-case theoretical approximation guarantees.
 \subsection{Illustrative example}
We demonstrate our result established in \S\ref{s:cardinality constrained sensor placement} on the benchmark electrical power grid, the IEEE 118-bus system. It consists of \(118\) buses, \(53\) power generators, and \(65\) power loads, connected to each other through network interconnections. A cyber-physical model of the generators and the loads proposed in \cite{ref:IliXieKhaMou-10} is adopted, where a Taylor linearization is performed at the nominal operating point to obtain a linear system. The obtained linear system \(G(A)\) has total number of state vertices equal to \(407\) and the total number of edges between them is \(920\). Table \ref{table:desription} describes the state variables of generators and loads of the IEEE \(118\)-bus system. 

\begin{table}[htbp]
\caption{State variables of the generators and loads}
\centering
 \begin{tabular}{c|l}
State variable & \multicolumn{1}{|c}{Description}\\
 \hline
 \(P_{T_G}\) & mechanical power of turbine\\
 \(P_G\) & electrical power of generator\\
 \(w_G\) & generator's output frequency\\
 \(a_G\) & valve opening of generator \\
 \hline
 \(P_L\) & electrical power delivered to load\\
 \(w_L\) & frequency measured at load\\
 \(I_L\) & real power consumed by load\\
\hline
 \end{tabular}
 \label{table:desription}
\end{table}

Assume that a transmission line \((i,j)\) exists between the generator \(i\) and load \(j\), and is represented by a digraph shown in Fig.~\ref{f:genload}. The frequency component \(w_{L_j}\) of bus \(j\) influences the dynamics of the power component \(P_{G_i}\) of bus \(i\) and vice-versa. This shows that we have outgoing edges from the frequencies into the powers of the components in the neighbouring buses.
\tikzset{middlearrow/.style={
        decoration={markings,
            mark= at position 0.5 with {\arrow{#1}} ,
        },
        postaction={decorate}
    }
} 
\begin{figure}[h!]
\begin{center}
\resizebox{1.8 in}{1.5 in}{
\begin{tikzpicture}
\node[] (0) at (0,0) {\(a_{G_i}\)};
\node[] (1) at (1.5, 0) {\(\omega_{G_i}\)};
 \node[] (2) at (3,0) {\(P_{G_i}\)};
 \node[] (3) at (3,-1.5) {\(P_{L_j}\)};
 \node[] (4) at (1.5,-1.5) {\(\omega_{L_j}\)};
 \node[] (5) at (0,-1.5) {\(I_{L_j}\)};
  \node[] (6) at (0.7,0.7) {\(P_{T_{G_i}}\)};
  \draw (1) edge[->,draw=red, line width=1.0pt] (3) (4) edge[->,draw=red, line width=1.0pt] (2) (5) edge[->,draw=blue, line width=1.0pt] (4); 
      \draw[bend left,middlearrow={>},draw=blue, line width=1.0pt] (0) to node [auto] {} (1);
   \draw[bend left,middlearrow={>},draw=blue, line width=1.0pt] (1) to node [auto] {} (0);
   \draw[bend left,middlearrow={>},draw=blue, line width=1.0pt] (1) to node [auto] {} (2);
  \draw[bend left,middlearrow={>},draw=blue, line width=1.0pt] (2) to node [auto] {} (1);
  \draw[bend left,middlearrow={>},draw=blue, line width=1.0pt] (3) to node [auto] {} (4);
  
   \draw[bend left,middlearrow={>},draw=blue, line width=1.0pt] (4) to node [auto] {} (3);
   
   \draw[bend left,middlearrow={>},draw=blue, line width=1.0pt] (0) to node [auto] {}
    (6);
    \draw[bend left,middlearrow={>},draw=blue, line width=1.0pt] (6) to node [auto] {} (1);
  \path(1) edge [loop below, line width=1.0pt] node {} (1);
    \path(2) edge [loop above, line width=1.0pt] node {} (2);
  \path(0) edge [loop left, line width=1.0pt] node {} (0);
  \path(4) edge [loop below, line width=1.0pt] node {} (4);
  \path(3) edge [loop below, line width=1.0pt] node {} (3);
  \path(5) edge [loop left, line width=1.0pt] node {} (5);
   \path(6) edge [loop above, line width=1.0pt] node {} (6);

    \end{tikzpicture}}
   \end{center}
   \caption{Illustration of the digraph representation of generator \(i\) connected to load \(j\) through a transmission line \((i, j)\). }
\label{f:genload}
\end{figure}

According to the construction shown in Fig.~\ref{f:genload}, each state vertex of the generators and the loads has a self-loop. Therefore, Assumption \ref{a:selfloop} is valid for \(G(A)\). We consider the given output matrix \(C=I_d\), where \(d=407\). We provide an approximate solution of the cardinality constrained sensor placement  Problem \eqref{e:submodularC(S)output} by employing Algorithm \ref{Algoobservabilitysubodularoutput}. Fig.~\ref{f:plot_states_output} depicts the variation in \(\Xi(\cdot)\) as the permissible number of outputs changes. If the number of  permitted outputs is small, then the maximum size of the set of states structurally observable is small, which is natural. However, beyond  a certain threshold of the permissible number of outputs, in this case \(14\), the set of all the vertices in \(G(A)\) become structurally observable. 
 
\begin{figure}[h!]
\centering
\includegraphics[width = 3.5 in, height=2 in]{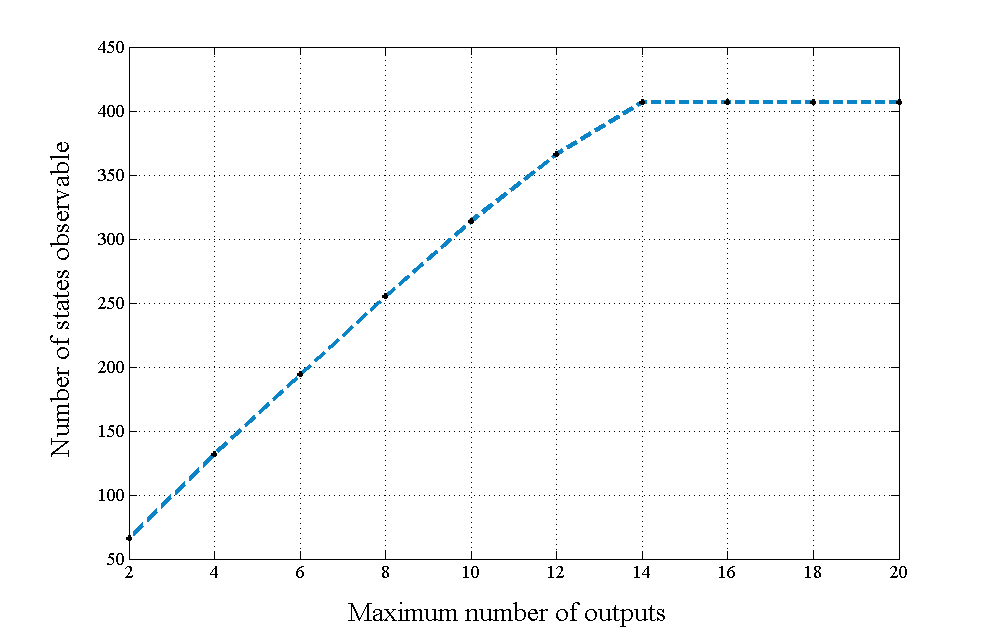}
\caption{Illustration of the change in the cardinality of the set of states structurally observable against the permissible number of outputs.}
\label{f:plot_states_output}
\end{figure}
\section{Concluding remarks}
\label{s:conclusion}
In this article we have addressed two problems related to optimal sensor placement in linear systems: the minimal sensor placement problem and the cardinality constrained sensor placement problem.

	\begin{itemize}[label=\(\circ\), leftmargin=*]
		\item We produced an efficient polynomial time algorithm to solve the minimal sensor placement problem when the desired bound on the structural observability index is \(2\). 
		\item We have demonstrated an interesting transition in the hardness of the minimal sensor placement problem as the desired bound changes from \(2\) to \(3\).
		\item The NP-hardness of the minimal sensor placement problem does not preclude the existence of classes of systems for which it is possible to determine solutions efficiently. In fact, we provided a linear time algorithm to solve this problem under a mild assumption that the system structure is a directed tree with self-loop at each state vertex.
		\item We proved that the cardinality constrained placement problem is a hard combinatorial optimization problem and remains computationally difficult for self-damped systems. We employed a simple greedy algorithm to find an \((1-\frac{1}{e})\)-approximate solution of this problem for self-damped systems.
	\end{itemize}

	By standard duality arguments, all our results have analogous counterparts and interpretations for controllability and actuator placement. Future work involves determining other interesting subclasses where the current problems can be solved efficiently, and identifying vulnerable connections between the states whose deletion leads to sudden jumps in the observability index of the system.%
\bibliographystyle{alpha}
\bibliography{ref}

\newcommand{\etalchar}[1]{$^{#1}$}
\begin{thebibliography}{WSY{\etalchar{+}}14}

\bibitem[AHCN07]{ref:AkuHayChiNg-07}
T.~Akutsu, M.~Hayashida, W.~Ching, and M.~K. Ng.
\newblock Control of boolean networks: Hardness results and algorithms for tree
  structured networks.
\newblock {\em Journal of theoretical biology}, 244(4):670--679, 2007.

\bibitem[CABP16]{ref:ClaAloBusPoo-16}
A.~{Clark}, B.~{Alomair}, L.~{Bushnell}, and R.~{Poovendran}.
\newblock {\em Submodularity in dynamics and control of networked systems}.
\newblock Communications and Control Engineering Series. Springer, Cham, 2016.

\bibitem[CBP12]{ref:ClaBusPoo-12}
A.~{Clark}, L.~{Bushnell}, and R.~{Poovendran}.
\newblock On leader selection for performance and controllability in
  multi-agent systems.
\newblock In {\em 2012 IEEE 51st IEEE Conference on Decision and Control
  (CDC)}, pages 86--93, 2012.

\bibitem[CDT08]{ref:ComDioTri-08}
C.~{Commault}, J.~{Dion}, and D.~H. {Trinh}.
\newblock Observability preservation under sensor failure.
\newblock {\em IEEE Transactions on Automatic Control}, 53(6):1554--1559, 2008.

\bibitem[CLRS09]{ref:CorLeiRivSte-09}
T.~H. {Cormen}, C.~E. {Leiserson}, R.~L. {Rivest}, and C.~{Stein}.
\newblock {\em Introduction to {A}lgorithms}.
\newblock MIT Press, Cambridge, MA, third edition, 2009.

\bibitem[CM13]{ref:ChaMes-13}
A.~{Chapman} and M.~{Mesbahi}.
\newblock On strong structural controllability of networked systems: A
  constrained matching approach.
\newblock In {\em 2013 American Control Conference}, pages 6126--6131, 2013.

\bibitem[DK13]{ref:DooKha-13}
M.~{Doostmohammadian} and U.~A. {Khan}.
\newblock On the genericity properties in distributed estimation: Topology
  design and sensor placement.
\newblock {\em IEEE Journal of Selected Topics in Signal Processing},
  7(2):195--204, 2013.

\bibitem[DRZK17]{ref:DooRabZarKha-17}
M.~{Doostmohammadian}, H.~R. {Rabiee}, H.~{Zarrabi}, and U.~A. {Khan}.
\newblock Distributed estimation recovery under sensor failure.
\newblock {\em IEEE Signal Processing Letters}, 24(10):1532--1536, 2017.

\bibitem[GJ79]{ref:GarJoh-79}
M.~R. Garey and D.~S. Johnson.
\newblock {\em Computers and {I}ntractability}.
\newblock W. H. Freeman and Co., San Francisco, Calif., 1979.
\newblock A {G}uide to the {T}heory of NP-{C}ompleteness, A Series of Books in
  the Mathematical Sciences.

\bibitem[Hoc96]{ref:Hoc-96}
D.~S. Hochbaum.
\newblock Approximating covering and packing problems: set cover, vertex cover,
  independent set, and related problems.
\newblock {\em Approximation Algorithms for NP-Hard Problem}, pages 94--143,
  1996.

\bibitem[IXKM10]{ref:IliXieKhaMou-10}
M.~D. Ilic, L.~Xie, U.~A. Khan, and J.~M.~F. Moura.
\newblock Modeling of future cyber-physical energy systems for distributed
  sensing and control.
\newblock {\em IEEE Transactions on Systems, Man, and Cybernetics - Part A:
  Systems and Humans}, 40(4):825--838, 2010.

\bibitem[KM77]{ref:KunMis-77}
S.~Kundu and J.~Misra.
\newblock A linear tree partitioning algorithm.
\newblock {\em SIAM Journal on Computing}, 6(1):151--154, 1977.

\bibitem[KPK{\etalchar{+}}18]{ref:KruPeqKarMouAgu-18}
S.~{Kruzick}, S.~{Pequito}, S.~{Kar}, J.~M.~F. {Moura}, and A.~P. {Aguiar}.
\newblock Structurally observable distributed networks of agents under cost and
  robustness constraints.
\newblock {\em IEEE Transactions on Signal and Information Processing over
  Networks}, 4(2):236--247, 2018.

\bibitem[KV06]{ref:BerVyg-06}
B.~Korte and J.~Vygen.
\newblock {\em Combinatorial {O}ptimization}, volume~21 of {\em Algorithms and
  Combinatorics}.
\newblock Springer-Verlag, Berlin, third edition, 2006.
\newblock Theory and algorithms.

\bibitem[LB16]{ref:LiuBarAlb-16}
Y.~Y. Liu and A.~L. Barab{\'a}si.
\newblock Control principles of complex systems.
\newblock {\em Reviews of Modern Physics}, 88(3):035006, 2016.

\bibitem[Lin74]{ref:Lin-74}
C.~T. Lin.
\newblock Structural controllability.
\newblock {\em IEEE Transactions on Automatic Control}, 19(3):201--208, 1974.

\bibitem[Lov83]{ref:Lov-83}
L.~Lov\'asz.
\newblock Submodular functions and convexity.
\newblock In {\em Mathematical Programming: The State of the Art ({B}onn,
  1982)}, pages 235--257. Springer, Berlin, 1983.

\bibitem[May01]{ref:May-01}
R.~M. May.
\newblock {\em Stability and complexity in model ecosystems}, volume~6.
\newblock Princeton university press, 2001.

\bibitem[Mor82]{ref:HMor-82}
H.~Mortazavian.
\newblock On {$k$}-controllability and {$k$}-observability of linear systems.
\newblock In {\em Analysis and optimization of systems ({V}ersailles, 1982)},
  volume~44 of {\em Lecture Notes in Control and Information Sciences}, pages
  600--612. Springer, Berlin, 1982.

\bibitem[NPP16]{ref:NowPrePap-16}
C.~{Nowzari}, V.~M. {Preciado}, and G.~J. {Pappas}.
\newblock Analysis and control of epidemics: A survey of spreading processes on
  complex networks.
\newblock {\em IEEE Control Systems Magazine}, 36(1):26--46, 2016.

\bibitem[NW99]{ref:NemWol-99}
G.~Nemhauser and L.~Wolsey.
\newblock {\em Integer and Combinatorial Optimization}.
\newblock Wiley-Interscience Series in Discrete Mathematics and Optimization.
  John Wiley \& Sons, Inc., New York, 1999.
\newblock Reprint of the 1988 original, A Wiley-Interscience Publication.

\bibitem[NWF78]{ref:NemWolFis-78}
G.~L. Nemhauser, L.~A. Wolsey, and M.~L. Fisher.
\newblock An analysis of approximations for maximizing submodular set
  functions. {I}.
\newblock {\em Math. Programming}, 14(3):265--294, 1978.

\bibitem[Ols14]{ref:AOls-14}
A.~Olshevsky.
\newblock Minimal controllability problems.
\newblock {\em IEEE Transactions on Control of Network Systems}, 1(3):249--258,
  2014.

\bibitem[PKA14]{ref:PeqKarAgu-14}
S.~{Pequito}, S.~{Kar}, and A.~P. {Aguiar}.
\newblock Minimum number of information gatherers to ensure full observability
  of a dynamic social network: A structural systems approach.
\newblock In {\em 2014 IEEE Global Conference on Signal and Information
  Processing (GlobalSIP)}, pages 750--753, 2014.

\bibitem[PKA16]{ref:PeqKarAgui-16}
S.~Pequito, S.~Kar, and A.~P. Aguiar.
\newblock A framework for structural input/output and control configuration
  selection in large-scale systems.
\newblock {\em IEEE Transactions on Automatic Control}, 61(2):303--318, 2016.

\bibitem[PKK{\etalchar{+}}13]{ref:PeqKruKarMouPed-13}
S.~{Pequito}, S.~{Kruzick}, S.~{Kar}, J.~M.~F. {Moura}, and A.~{Pedro Aguiar}.
\newblock Optimal design of distributed sensor networks for field
  reconstruction.
\newblock In {\em 21st European Signal Processing Conference (EUSIPCO 2013)},
  pages 1--5, 2013.

\bibitem[PPBP17]{ref:SerAlbGeo-17}
S.~Pequito, V.~M. Preciado, A.~L. Barab{\'a}si, and G.~J. Pappas.
\newblock Trade-offs between driving nodes and time-to-control in complex
  networks.
\newblock {\em Scientific Reports}, 7:39978, 2017.

\bibitem[RB05]{ref:RenBea-05}
W.~Ren and R.~W. Beard.
\newblock Consensus seeking in multiagent systems under dynamically changing
  interaction topologies.
\newblock {\em IEEE Transactions on Automatic Control}, 50(5):655--661, 2005.

\bibitem[RBA07]{ref:RenBeaAtk-07}
W.~{Ren}, R.~W. {Beard}, and E.~M. {Atkins}.
\newblock Information consensus in multivehicle cooperative control.
\newblock {\em IEEE Control Systems Magazine}, 27(2):71--82, 2007.

\bibitem[Rei88]{ref:Rei-88}
K.~J. Reinschke.
\newblock {\em Multivariable Control: a Graph-Theoretic Approach}, volume 108
  of {\em Lecture Notes in Control and Information Sciences}.
\newblock Springer-Verlag, Berlin, 1988.

\bibitem[SCL16]{ref:SumCorLyg-16}
T.~H. Summers, F.~L. Cortesi, and J.~Lygeros.
\newblock On submodularity and controllability in complex dynamical networks.
\newblock {\em IEEE Transactions on Control of Network Systems}, 3(1):91--101,
  2016.

\bibitem[SDT97]{ref:SueDau-97}
C.~Sueur and G.~Dauphin-Tanguy.
\newblock Controllability indices for structured systems.
\newblock {\em Linear Algebra and its Applications}, 250:275--287, 1997.

\bibitem[SH13]{ref:SunHad-13}
S.~Sundaram and C.~N. Hadjicostis.
\newblock Structural controllability and observability of linear systems over
  finite fields with applications to multi-agent systems.
\newblock {\em IEEE Transactions on Automatic Control}, 58(1):60--73, 2013.

\bibitem[TJP16]{ref:TzoJadPapp-16}
V.~{Tzoumas}, A.~{Jadbabaie}, and G.~J. {Pappas}.
\newblock Sensor placement for optimal kalman filtering: Fundamental limits,
  submodularity, and algorithms.
\newblock In {\em 2016 American Control Conference (ACC)}, pages 191--196,
  2016.

\bibitem[WSY{\etalchar{+}}14]{ref:WanSyeYinPanZha-14}
L.~Y. Wang, A.~Syed, G.~G. Yin, A.~Pandya, and H.~Zhang.
\newblock Control of vehicle platoons for highway safety and efficient utility:
  consensus with communications and vehicle dynamics.
\newblock {\em J. Syst. Sci. Complex.}, 27(4):605--631, 2014.

\bibitem[ZLL{\etalchar{+}}17]{ref:ZheLiBorHed-17}
Y.~{Zheng}, S.~E. {Li}, K.~{Li}, F.~{Borrelli}, and J.~K. {Hedrick}.
\newblock Distributed model predictive control for heterogeneous vehicle
  platoons under unidirectional topologies.
\newblock {\em IEEE Transactions on Control Systems Technology},
  25(3):899--910, 2017.

\end{thebibliography}


\bigskip
\bigskip
\end{document}